%% file: Langmuir.tex
\documentclass[11pt]{elsarticle}

\usepackage[hidelinks]{hyperref}
\usepackage[utf8]{inputenc}
\usepackage[T1]{fontenc}
\usepackage[left = 1.0in, right = 1.0in, top = 1.0in, bottom = 1.0in]{geometry}
\usepackage{graphicx}
\usepackage{bbm}
\usepackage{stmaryrd}
\usepackage{xcolor}
\usepackage{amsmath, amsfonts, amsthm, amssymb}

\theoremstyle{plain}
	\newtheorem{theorem}{Theorem}[section]
	\newtheorem{proposition}[theorem]{Proposition}    
	\newtheorem{lemma}[theorem]{Lemma}          
	\newtheorem{corollary}[theorem]{Corollary}

\theoremstyle{definition}
	\newtheorem{definition}[theorem]{Definition}
	\newtheorem{remark}[theorem]{Remark}

\DeclareMathOperator{\meas}{meas}

 \def\NN{{\mathbb N}}  
 \def\QQ{{\mathbb Q}} \def\RR{{\mathbb R}}

\def\cD{{\mathcal D}}   
   
\def\cJ{{\mathcal J}}   
 \def\cN{{\mathcal N}} \def\cO{{\mathcal O}} 
\def\cP{{\mathcal P}}   
   
\def\cV{{\mathcal V}} \def\cW{{\mathcal W}}

  \def\mR{{\mathfrak R}} 
  \def\mU{{\mathfrak U}} 
\def\mV{{\mathfrak V}}

\begin{document}

\begin{frontmatter}
\title{Existence of solutions for a bi-species kinetic model of a cylindrical Langmuir probe}    
\author{M. Badsi}
\author{L. Godard-Cadillac}
\address{Nantes Université, Laboratoire de Mathématiques Jean Leray, \fnref{label3}
2 Chemin de la Houssinière BP 92208,\\ 44322 Nantes Cedex 3} 
\begin{abstract}
In this article, we study a collisionless kinetic model for plasmas in the neighborhood of a cylindrical metallic Langmuir probe. 
This model consists in a bi-species Vlasov-Poisson equation in a domain contained between two cylinders with prescribed boundary conditions. 
The interior cylinder models the probe while the exterior cylinder models the interaction with the plasma core. We prove the existence of a weak-strong solution for this model in the sense that we get a weak solution for the 2 Vlasov equations and a strong solution for the Poisson equation. 
The first parts of the article are devoted to explain the model and proceed to a detailed study of the Vlasov equations. This study then leads to a reformulation of the Poisson equation as a 1D non-linear and non-local equation and we prove it admits a strong solution using an iterative fixed-point procedure.
 \end{abstract}
\end{frontmatter}

\section*{Introduction}
The Langmuir probe is a measurement device that is used to determine the local properties of a plasma such as its density, temperature and plasma potential known as plasma parameters. It is used in a wide range of applications. In practice, to determine the plasma parameters, the probe voltage is varied within a sufficiently large range and the collected current is recorded. The curve of the collected current versus the applied probe voltage is called the characteristic of the probe. It is the main object of interest in the probe modeling theory. The modeling of probes has been the aim of a lot of physical theories and several works aim at studying in detail these theories (see for instance~\cite{langmuir,Allen, Bernstein-Rabinowitz}). For a kinetic modeling of the Langmuir probe, we refer the reader to the monograph of Laframboise \cite{Laframboise} for a  general overview where both cylindrical and spherical probe models based on the stationary Vlasov-Poisson equations are proposed. Some discussions on the particles orbits and numerical simulations can also be found.

At the mathematical level,  existence theories for kinetic equations modeling plasma particles interacting with a probe in a two dimensional setting is not well-known. There is nevertheless several results concerning stationary solutions for the Vlasov-Poisson equations. The more relevant within the context of probe is the work of Greengard and Raviart \cite{Raviart} which deals with the one dimensional stationary solutions of Vlasov-Poisson boundary value problem where a very complete analysis of particles trajectories is made. An extension of this work by Degond and al to the case of a cylindrically symmetric diode can be found in \cite{Degond-Raviart-Poupaud-Jaffard}. On the contrary to the model that we study here,  their work considers one species of particles and the analysis of existence uses a maximum principle for the Poisson equation. Our approach is different and based on explicit expression of the macroscopic densities. This approach gives a good understanding of the trajectories of the particles and of the effective electrical potential as it is a constructive approach. This is also of particular interest in view of the numerical simulations. We also mention the work of Bernis \cite{Bernis} which is concerned with the existence of stationary solution with cylindrical symmetry for the Vlasov-Poisson equations in the whole space. Others works on stationary Vlasov-Poisson equations can be found in the non exhaustive list \cite{Knopf,Pokhozhaev,Rein-stationary,Belyaeva,Badsi_Campos-Pinto_Depres_Godard-Cadillac_2021}.

In this work, we consider the modeling of a cylindrical probe immersed in a plasma made of one species of ions and of electrons and its analysis. 
We use a collisionless kinetic description to model the transport of particles under the action of the self consistent electric potential. 
The unknown are assumed to obey the stationary Vlasov-Poisson equations written in polar coordinates. 
To model the interaction with the probe, we assume that particles are emitted from the core plasma while at the probe particles are absorbed. 
The probe potential is fixed to some arbitrary value while in the plasma the electric potential is taken equal to a reference potential value. 
To construct weak solutions of the Vlasov equation,  we use the method of characteristics and the conservation of the local energy and angular momentum to decompose the phase space for each species of particles. 
This decomposition of the phase space yields the definition of two distinct regions : one corresponds to trajectories of particles that reach the probe, the other one corresponds to trajectories that do not reach the probe. 
Because this decomposition is made in full generality, it introduces the study of the potential barrier (both its height and position) that separates the trajectories of the particles that reach the probe from the others. 
On closed trajectories (not connected to the boundaries), our solution is taken to be zero though it could be any other distribution function. 

The study of these different regions of the phase space eventually gives a compact reformulation of the source term in the Poisson equation that involves non-linear and non-local terms.
To deal with non-local terms, the strategy consists first in replacing them by parameters. In such a situation, the existence of a solution follows by standard variational arguments.
In a second time, we adjust these parameters in such a way that we can recover the initial non-local equation. We proceed by using a fixed-point procedure so that the parameters are expected to converge towards the associated terms.
The main technical difficulty lays in obtaining the convergence of the solution itself during this fixed-point procedure. 
The convergence is obtained using  three main ingredients:  a  general $L^{\infty}$ estimate on the macroscopic density that is uniform in the electric potential, a Hölder estimate on the non-linear term and continuity properties on the non-local terms. 
These estimates are obtained provided the incoming distribution functions obey some appropriate integrability properties in velocities which is reminiscent of the work of \cite{Raviart}. 
The obtained sequence is then proved to converge towards a solution of the original problem. The qualitative description of the solution and its numerical simulation will be the purpose of a future work.

\section{Modeling the probe}
We consider a non collisional and unmagnetized plasma made of one species of ions and of electrons in which is immersed a cylindrical probe. The radius of the probe is $r_{p} > 0$ and the length of its axis is $L > 0.$ We assume $L \gg r_{p}$ so that an invariance along the probe axis is assumed. Then, we only model the planar motion of particles in the open set $\Omega = \lbrace (x,y) \in \RR^2 \: : \: r_{p}^2 <x^2+y^2 < r_{b}^2 \rbrace$ where $r_{b} >r_{p}$ is an outer boundary radius (see Figure~\ref{fig:probe_drawing}). Outside the radius $r_b$ lays the plasma core.
\begin{figure}[h!]
\center
 \scalebox{0.4}{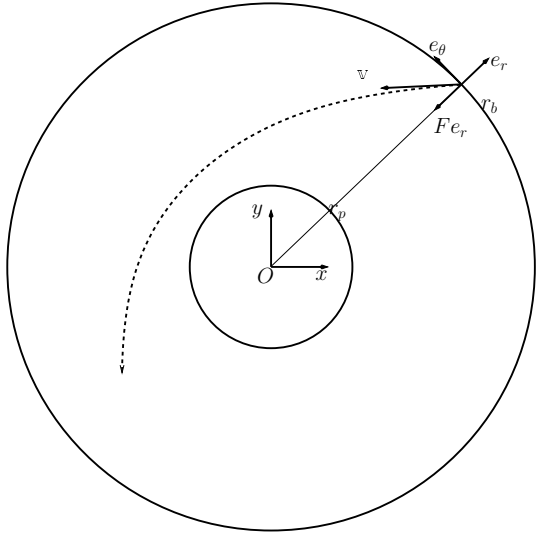 } 
 \caption{Sketch of a trajectory of a particle into a radial force field entering at $r = r_{b}$ with a velocity $\mathbbm{v}$. }\label{fig:probe_drawing}
\end{figure}
\subsection{The Vlasov-Poisson equations in polar coordinates}
 In cartesian coordinates, particles positions are denoted $\mathbbm{x} := (x,y)$ and velocities are denoted $\mathbbm{v} := (v_{x},v_{y})$. In polar coordinates, particles positions write $\mathbbm{x} = (x,y) = r \mathbbm{e}_{r}$ with $r = \sqrt{x^2+y^2}$ and $\mathbbm{e_{r}} = (\cos \theta, \sin \theta)$, and particles velocities write $\mathbbm{v} := (v_{x},v_{y}) = v_{r} \mathbbm{e}_{r} + v_{\theta} \mathbbm{e}_{\theta}$ with $v_{r} = \mathbbm{v} \cdot \mathbbm{e}_{r}$, $v_{\theta} = \mathbbm{v} \cdot \mathbbm{e}_{\theta}$ and $\mathbbm{e_{\theta}} = (-\sin \theta, \cos \theta)$. The unknown are the non negative particles distribution functions of ions and electrons in the phase space $(r,v_{r},v_{\theta}) \in [r_{p},r_{b}] \times \RR^2$ and the electrostatic potential. They are denoted $f_{i}(r,v_{r},v_{\theta})$, $f_{e}(r,v_{r},v_{\theta})$ and  $\phi(r)$. They are assumed to obey the Vlasov-Poisson equations which in polar coordinates write:
\begin{align}
v_r\, \partial_{r} f_i - \frac{v_{r} v_{\theta}}{r} \partial_{v_{\theta}}f_{i} + \left( \frac{v_{\theta}^2}{r} - \frac{q}{m_{i}} \partial_{r} \phi \right) \,\partial_{v_{r}} f_{i} = 0, \quad  \forall (r,v_{r},v_{\theta}) \in (r_{p},r_{b}) \times \RR^2 \label{Vlasov-i-1}\\
v_r\, \partial_{r} f_e - \frac{v_{r} v_{\theta}}{r} \partial_{v_{\theta}}f_{e} + \left( \frac{v_{\theta}^2}{r} + \frac{q}{m_{e}} \partial_{r} \phi \right) \,\partial_{v_{r}} f_{e}= 0, \quad \forall (r,v_{r},v_{\theta}) \in (r_{p},r_{b}) \times \RR^2 \label{Vlasov-e-1} \\
-\frac{1}{r} \frac{d}{dr}\bigg(r \frac{d\phi}{dr}\bigg)(r) = \frac{q}{\varepsilon_{0}} \int_{\RR^2}\Big(f_{i}(r,v_{r},v_{\theta}) - f_{e}(r,v_{r},v_\theta) \Big)dv_{r}\, dv_{\theta}, \quad \forall r \in (r_{p},r_{b})\label{Poisson-1},
\end{align}
where $q > 0$ is the electrical elementary charge, $\varepsilon_{0} > 0$ is the vacuum electrical permittivity and $m_{i} > m_{e}  > 0$ are respectively the mass of one ion and of one electron.
Equations \eqref{Vlasov-i-1}-\eqref{Poisson-1} model the transport of the charged particles under the action of the self-consistent electrostatic potential.
For the sake of conciseness, we denote for all $r \in [r_{p},r_{b}]$ the ions and electrons macroscopic charge densities by:
\begin{align}
n_{i}(r) = q\int_{\RR^2} f_{i}(r,v_{r},v_{\theta})\, d v_{r}\, d v_{\theta}, \quad n_{e}(r) = q\int_{\RR^2} f_{e}(r,v_{r},v_{\theta})\, dv_{r}\, dv_{\theta}.
\end{align}
In the context of the Langmuir probe theory \cite{Laframboise,langmuir} the radial current density is an important quantity to be computed. For each species $s = i,e$ and all $r \in [r_{p},r_{b}]$ it is defined by:
\begin{align}
\mathbbm{J}_{s}(r) = j_{s}(r)\, \mathbbm{e}_{r},
\end{align}

\begin{align}
j_{s}(r) := q \int_{\RR^2} f_{i}(r,v_r,v_{\theta}) \,v_{r}\, dv_{r}\, dv_{\theta}.
\end{align}

\subsection{Boundary conditions in the plasma and at the probe} 
We assume that far away from the outer boundary radius $ r > r_{b}$ there exists a ionizing source of particles (the plasma core) that makes both ions and electrons enter at $r = r_{b}$. We model these incoming particles from the plasma core by the following boundary condition
\begin{align}
  \forall (v_{r},v_{\theta}) \in \RR^{-}_{*} \times \RR, \qquad   f_{i}(r_{b},v_{r},v_{\theta}) = f_{i}^{b}(v_{r},v_{\theta}), \quad f_{e}(r_{b},v_{r},v_{\theta}) = f_{e}^{b}(v_{r},v_{\theta}), \label{bc-plasma-1}
\end{align}
where  $f_{i}^{b} : \RR^{-}_{*} \times \RR \rightarrow \RR^{+}$ and $f_{e}^{b} : \RR^{-}_{*} \times \RR \rightarrow \RR^{+}$ denote arbitrarily given distribution functions of the incoming particles. These functions must be assumed to be symmetric with respect to the angular velocity $v_{\theta}$ to ensure the absence of ortho-radial current and the invariance with respect to the angular variable. The zero potential reference is taken to be at $r = r_{b}$:
\begin{equation}\label{bc-phi-1}
    \phi(r_b)=0.
\end{equation}
We assume the probe to be non-emitting, that is at $r = r_{p}$, no particles are emitted in the direction to the plasma. We also consider that the potential of the probe is fixed at a value $\phi_{p} \in \RR$.  The boundary conditions at $r=r_p$ then write 
\begin{align}
\forall (v_{r},v_{\theta}) \in \RR^{+}_{*} \times \RR, \quad f_{i}(r_{p},v_{r},v_{\theta}) = 0, \quad  f_{e}(r_{p},v_{r},v_{\theta}) = 0, \label{bc-probe-1}
\end{align}
\begin{align}
\phi(r_{p}) = \phi_{p}. \label{bc-phi-2}
\end{align}

\begin{remark} Since $f_{i}^{b}(v_{r},v_{\theta})$ and $f_{e}^{b}(v_{r},v_{\theta})$ are both symmetric with respect to $v_{\theta}$ then the solutions of the Vlasov equations \eqref{Vlasov-i-1} and \eqref{Vlasov-e-1} are also symmetric with respect to $v_{\theta}$. There is not any ortho-radial current: $\int_{\RR^2} f_{s}(r,v_{r},v_{\theta})v_{\theta} dv_{r} dv_{\theta} = 0,$  for each species $s = i,e$ and for all $r \in [r_{p},r_{b}]$.
\end{remark}

\subsection{Dimensionless equations}
Consider the following physical constants $\lambda =  \sqrt{\varepsilon_{0} k_b T_{e}/(q^2 N_{0})}$ (Debye length) and $c_{s} = $ $\sqrt{k_b T_{e}/m_{i}}$ (ions acoustic speed)
where $T_{e} \gg T_{i}$ is a reference electron temperature, $N_{0} > 0$ is a reference plasma density and $k_b$ denotes the Boltzmann constant. We define the rescaled variables
\begin{align}
\hat{r} = \frac{r}{r_{p}}, \quad \hat{v_{r}} = \frac{v_{r}}{c_{s}}, \quad \hat{v_{\theta}} = \frac{v_{\theta}}{c_{s}}.
\end{align}
We also define the rescaled particles distribution functions and the rescaled electrostatic potential
\begin{align}
\hat{f_{i}}(\hat{r},\hat{v_{r}},\hat{v_{\theta}}) = \frac{c_s^2}{N_{0}} f_{i}(r,v_r,v_{\theta}), \quad \hat{f_{e}}(\hat{r},\hat{v_{r}},\hat{v_{\theta}}) = \frac{c_s^2}{N_{0}} f_{e}(r,v_r,v_{\theta}), \quad \hat{\phi}(\hat{r}) =  \frac{q \phi(r)}{k_bT_{e}}.
\end{align}

The rescaled unknown verify the dimensionless Vlasov-Poisson equations which after dropping the dimensionless notation $\hat{.}$  write:
\begin{align}
&v_r\, \partial_{r} f_i - \frac{v_{r} v_{\theta}}{r} \partial_{v_{\theta}}f_{i} + \left( \frac{v_{\theta}^2}{r} - \partial_{r} \phi \right) \,\partial_{v_{r}} f_{i} = 0, \qquad  \forall (r,v_{r},v_{\theta}) \in (1,r_{b}) \times \RR^2, \label{Vlasov-i}\\
&v_r\, \partial_{r} f_e - \frac{v_{r} v_{\theta}}{r} \partial_{v_{\theta}}f_{e} + \left( \frac{v_{\theta}^2}{r} + \frac{1}{\mu} \partial_{r} \phi \right) \,\partial_{v_{r}} f_{e} = 0, \qquad \forall (r,v_{r},v_{\theta}) \in (1,r_{b}) \times \RR^2, \label{Vlasov-e} \\
&-\frac{\overline{\lambda}^2}{r}\frac{d}{dr}\bigg(r\,\frac{d\phi}{dr}\bigg)(r) = n_{i}(r)-n_{e}(r), \qquad \forall r \in (1,r_{b})\label{Poisson},
\end{align}

where $\mu = m_{e}/m_{i}$ is the mass ratio and $\overline{\lambda} = \lambda/r_{p}$ is a normalized Debye length.
An additional re-scaling of the velocities space for the electronic Vlasov equation \eqref{Vlasov-e} is given by the change of variables and unknown
\begin{align}
    \hat{v_{r}} = \sqrt{\mu}\, v_{r}, \quad \hat{v_{\theta}} = \sqrt{\mu}\, v_{\theta}, \quad \hat{f}_{e}(r,\hat{v_{r}},\hat{v_{\theta}}) = \mu\, f_{e}(r,v_{r},v_{\theta})
\end{align}
which yields again after dropping the notation $\hat{.}$ the same Vlasov equation \eqref{Vlasov-e} with $\mu = 1$. In the Poisson equation \eqref{Poisson},
the dimensionless macroscopic densities are then given by
\begin{align}
n_{i}(r) = \int_{\RR^2} f_{i}(r,v_r,v_{\theta})\, dv_{r}\, dv_{\theta}, \quad n_{e}(r) = \int_{\RR^2} f_{e}(r,v_{r},v_{\theta})\, dv_{r}\, dv_{\theta}
\end{align}
and the dimensionless radial currents are given by
\begin{align}
j_{i}(r) = \int_{\RR^2} f_{i}(r,v_{r},v_{\theta})\, v_{r}\, dv_{r}\, dv_{\theta}, \quad j_{e}(r) = \frac{1}{\sqrt{\mu}} \int_{\RR^2} f_{e}(r,v_{r},v_{\theta})\, v_{r}\, dv_{r}\, dv_{\theta}.
\end{align}
The  factor $1/\sqrt{\mu}$ is natural in view of the difference of mobility between ions and electrons.
The obtained problem is 
\begin{equation} \label{VPBVP}
    \begin{cases}
    &v_r\, \partial_{r} f_i - \frac{v_{r} v_{\theta}}{r} \partial_{v_{\theta}}f_{i} + \left( \frac{v_{\theta}^2}{r} - \partial_{r} \phi \right) \,\partial_{v_{r}} f_{i} = 0,\qquad \forall (r,v_{r},v_{\theta}) \in (1,r_{b}) \times \RR^2,\\
    &v_r\, \partial_{r} f_e - \frac{v_{r} v_{\theta}}{r} \partial_{v_{\theta}}f_{e} + \left( \frac{v_{\theta}^2}{r} +  \partial_{r} \phi \right) \,\partial_{v_{r}} f_{e} = 0, \qquad \forall (r,v_{r},v_{\theta}) \in (1,r_{b}) \times \RR^2,\\
    &-\frac{\overline{\lambda}^2}{r}\frac{d}{dr}\Big(r\,\frac{d\phi}{dr}\Big)(r) = n_{i}(r)-n_{e}(r), \qquad \forall r \in (1,r_{b}),\\
    &f_{i}(r_{b},v_{r},v_{\theta}) = f_{i}^{b}(v_{r},v_{\theta}), \quad f_{e}(r_{b},v_{r},v_{\theta}) = f_{e}^{b}(v_{r},v_{\theta}),\quad \forall (v_{r},v_{\theta}) \in \RR^{-}_{*} \times \RR, \quad \\
    &f_{i}(r_{p},v_{r},v_{\theta}) = 0, \quad  f_{e}(r_{p},v_{r},v_{\theta}) = 0, \quad \forall (v_{r},v_{\theta}) \in \RR^{+}_{*} \times \RR\\
    &\phi(r_{p}) = \phi_{p}, \quad \phi(r_{b}) = 0.
    \end{cases}
\end{equation}

Since in the proof of the existence of solutions the physical parameter $\overline{\lambda}$ is of little interest,  we consider in the following $\overline{\lambda} = 1.$
We nevertheless mention that in the qualitative description of the solutions the physical regime $\overline{\lambda}$  small is important  because a boundary layer known as the Debye sheath exists in the vicinity of the probe. 
See for instance \cite{Bohm,Riemann, Laframboise,badsi_krm} for further physical and mathematical details. 

\section{Main result}
We first define the notion of solutions that we consider for the Vlasov-Poisson equations with the boundaries and then state our main result. In this regard, we need some notations, we introduce the set of outgoing particles, the set of incoming particles:
\begin{align*}
    \Sigma^{\textnormal{out}} := \big( \lbrace r_{b} \rbrace \times \RR^{+} \times \RR \big) \cup \big( \lbrace 1 \rbrace \times \RR^{-} \times \RR\big), \quad \Sigma^{\textnormal{inc}} := \big(\lbrace r_{b} \rbrace \times \RR^{-}_{*} \times \RR\big)\cup\big(\lbrace{1\rbrace} \times \RR^{+}_\ast \times \RR\big)
\end{align*}
and denote the domain of work $Q := (1,r_{b}) \times \RR^2$. Observe that $\Sigma^{\textnormal{out}}=\partial Q\setminus\Sigma^{\textnormal{inc}}$. Define also $$
\mu_{s}:=
\begin{cases} 
1 \text{ if } s = i,\\
-1 \text{ if } s = e.
\end{cases}
$$
Solutions of the Vlasov equations with boundaries are not necessarily classical even though the incoming boundary data $f_{i}^{b}$ and $f_{e}^{b}$ are smooth. 
This is due to the geometry of the characteristic curves (they are defined in section \ref{sec:linear_vlasov}) and the boundary conditions \eqref{bc-plasma-1},\eqref{bc-probe-1}. A discontinuity in the solution at the boundary can occur and be propagated by the characteristics into the interior of the domain.
Therefore, we shall generically consider weak solutions for the Vlasov equations.
\begin{definition}[Weak solution to Vlasov equation] \label{def_weak_sol_vlasov} Let $\phi \in W^{1,\infty}(1,r_{b})$. Let $s = i,e$. Let $f_{s}\in L^1(Q)$ and $f_{s}^{b}\in L^1(\Sigma^{\textnormal{inc}})$. We say that $f_{s}$ is a weak solution of the Vlasov equation with the boundary condition $f_{s}^{b}$ if for every $\psi \in C^{1}\left( \overline{Q} \right)$ compactly supported on $\overline{Q}$ and such that $\psi_{| \Sigma^{\textnormal{out}}} = 0$, the following equality holds:
\begin{align}\label{eq:weak Vlasov}
     \int_1^{r_b} \int_{-\infty}^{+\infty}\int_0^{+\infty} &f_{s}(r,v_r,v_\theta)\Psi(r,v_r,v_\theta)\,dv_r\,dv_\theta\,dr \\&= \int_{-\infty}^{+\infty}\int_{-\infty}^{0}f_{s}^{b}(v_r,v_\theta)\,\psi(r_b,v_r,v_\theta)\, v_r\,dv_r\,dv_\theta
\end{align}
where
\begin{align*}
    \Psi(r,v_{r},v_{\theta}) = & v_{r} \partial_{r} \psi(r,v_{r},v_{\theta}) +  \left( \frac{v_{\theta}^{2}}{r}-\mu_{s}\partial_{r} \phi(r) \right) \partial_{v_{r}} \psi(r,v_{r},v_{\theta}) \\&- \frac{v_{r}}{r} \partial_{v_{\theta}} (v_{\theta} \psi)(r,v_{r},v_{\theta}).
\end{align*}
\end{definition}
This weak formulation of the Vlasov equation~\eqref{eq:weak Vlasov} can be reformulated in terms of duality brackets:
\begin{equation*}
    \left<\Psi,f_s\right>_{L^\infty(Q),L^1(Q)}=\left<\big(v_r\,\psi_{| \Sigma^{\textnormal{inc}}}\big),f_s^b\right>_{L^\infty(\Sigma^{\textnormal{inc}}),L^1(\Sigma^{\textnormal{inc}})}
\end{equation*}

The solution for the studied Vlasov-Poisson problem are weak solutions for the Vlasov equation and point-wise solution for the Poisson equation:
\begin{definition}[Weak-strong solution of the Vlasov-Poisson problem]\label{def_sol} Let $\phi_{p} \in \RR$. Let $f_{i}^{b}$ and $f_{e}^{b}$ two integrable functions on $\Sigma^{\textnormal{inc}}.$ We say that a triplet $(f_{i},f_{e},\phi)$ is a weak-strong solution of the Vlasov-Poisson Langmuir problem \eqref{VPBVP} if:
\begin{itemize}
    \item $\phi \in W^{2,\infty}(1,r_{b})$ and $f_{i},f_{e}\in L^1(Q)$.
    \item $f_{i}$ and $f_{e}$ are weak solutions of the Vlasov equations in the sense of definition \ref{def_weak_sol_vlasov}. 
    \item $\phi$ satisfies the Poisson equation \eqref{Poisson} pointwise in $[1,r_b]$ and the Dirichlet boundary conditions \eqref{bc-phi-1}\eqref{bc-phi-2}.
\end{itemize}
\end{definition}
In the above definition the boundary data are assumed to be in $L^1$. The regularity $\phi \in W^{2,\infty}(1,r_{b})$ is sufficient to ensure the existence and uniqueness of the characteristics curves defined in section \eqref{sec:linear_vlasov}.

Concerning our main result, we make use for technical reasons of extra integrability conditions on the incoming fluxes. For that purpose we define the Banach space $L^1_L(L^\infty_w(w\,dw))$ as being the space of measurable functions of $\RR^2$ such that the following norm is finite: 
\begin{equation}\label{def:Lebesgue 1}
    \|f\|_{L^1_L(L^\infty_w(w\,dw))}:=\int_\RR\sup_{w\in\RR}\big|w\,f(w,L)\big|\,dL.
\end{equation}
We also define the Banach space $L^1_w(L^\infty_L\,;dw/|w|^\gamma)$ where $0<\gamma<1$ from the following norm:
\begin{equation}\label{def:Lebesgue 2}
    \|f\|_{L^1_w(L^\infty_L\,;dw/|w|^\gamma)}:=\int_\RR\sup_{L\in\RR}\big|f(w,L)\big|\,\frac{dw}{|w|^\gamma}.
\end{equation}
 Note that these two norms are finite if, for instance, we have the following estimate:
\begin{equation*}
    \forall\;(w,L)\in\RR^2,\qquad|f(w,L)|\leq\frac{1}{|w|+|L|^2+1}.
\end{equation*}

The main result of this article is the following:
\begin{theorem} Let $\phi_{p} \in \RR$. Let $f_{i}^{b}$ and $f_{e}^{b}$ be two non-negative integrable functions defined on $\RR_-\times\RR$ symmetrical for the second variable. Suppose moreover that, with $s=i,e$,
\begin{equation*}
    \|f_s^b\|_{L^1_L(L^\infty_w(w\,dw))}<+\infty\qquad\text{and}\qquad\|f_s^b\|_{L^1_w(L^\infty_L\,;dw/|w|^\gamma)}<+\infty.
\end{equation*}
for some $0<\gamma<1$. 

Then the Vlasov-Poisson problem \eqref{VPBVP} with boundary values $f_{i}^{b}$ and $f_{e}^{b}$ admits a solution in the sense of Definition \ref{def_sol}.
\end{theorem}

\section{The linear Vlasov equations} \label{sec:linear_vlasov}
We consider for this section only the linear Vlasov equations \eqref{Vlasov-i} and \eqref{Vlasov-e} where for now the potential $\phi$ is fixed independently of the influence of the particles. 
The aim of the work done in this section is to reformulate the Vlasov equations to reduce the initial problem to a non-linear 1D Poisson equation.
We assume that $\phi\in W^{2,\infty}(1,r_{b})$, so that its derivative is Lipschitz continuous.
\subsection{Ionic phase diagram}
The characteristics associated with the Vlasov equation \eqref{Vlasov-i} are the solutions to the ordinary differential equations
\begin{equation} \label{char-i-1}
\begin{cases}\displaystyle
\frac{d}{dt}r(t) = v_{r}(t),\\
\displaystyle\frac{d}{dt}v_{r}(t) = \frac{v_{\theta}(t)^2}{r(t)} - \frac{d\phi}{dr}(r(t)),\\
\displaystyle\frac{d}{dt}v_{\theta}(t) = \frac{-v_{r}(t)\, v_{\theta}(t)}{r(t)}.
\end{cases}
\end{equation}
Since $d\phi/dr$ is Lipschitz continuous, for each initial condition $(r_{0},v_{r,0},v_{\theta,0}) \in (1,r_{b}) \times \RR^2$, Equations~\eqref{char-i-1} admits a unique solution $(r,v_r,v_\theta) \in C^{1}( \left(t_{\textnormal{inc}}(r_0,v_{r,0},v_{\theta,0}), t_{\textnormal{out}}(r_0,v_{r,0},v_{\theta,0})\right);$ $ [1,r_{b}] \times \RR^2)$ where 
\begin{align*}
&t_{\textnormal{inc}}(r_0,v_{r,0},v_{\theta,0}):= \inf \lbrace t' \leq 0 \: : \: r(t) \in (1,r_{b})\: \:  \forall t \in (t',0) \rbrace, \quad\\ &t_{\textnormal{out}}(r_0,v_{r,0},v_{\theta,0}):= \sup  \lbrace t' \geq 0 \: : \: r(t) \in (1,r_{b}) \: \:  \forall t \in (0,t') \rbrace
\end{align*}
denote respectively the incoming time and the outgoing time of the characteristics in the interval $(1,r_{b})$. 
They can be either finite of infinite.
Additionally, one has two constants of motion:  the total energy and the angular momentum. 
Indeed, the characteristics satisfy for all $t \in \left(t_{\textnormal{inc}}(r_0,v_{r,0},v_{\theta,0}), t_{\textnormal{out}}(r_0,v_{r,0},v_{\theta,0})\right),$
\begin{align*}
&\frac{d}{dt} \left(  \frac{v_{r}^{2}(t)+ v_{\theta}^2(t)}{2} + \phi(r(t)) \right)= 0, \\
&\frac{d}{dt} \left( r(t)v_{\theta}(t) \right) = 0.
\end{align*}
Therefore the characteristics are contained in the following level sets defined for $L \in \RR$ and $e \in \RR$ by
\begin{align*}
\mathcal{C}_{L,e} := \Big\lbrace (r,v_{r},v_{\theta}) \in (1,r_{b}) \times \RR^2 \: : \:  r v_{\theta} = L \quad\text{and}\quad \frac{v_r^{2} + v_{\theta}^2}{2} + \phi(r) = e \Big\rbrace.
\end{align*}
These sets give a description of the phase space according to the values of $L$ and $e.$ In this regard, it is convenient to introduce for $L \in \RR$ the effective potential defined by
\begin{align}
\forall r \in [1,r_{b}] \quad U_{L}(r) := \frac{L^2}{2r^2} + \phi(r).  \label{U_L}
\end{align}
Since $U_{L}$ is a continuous function on $[1,r_{b}]$, it reaches its maximum value at some point in $[1,r_{b}].$ Its maximum value is denoted 
\begin{align*}
\overline{U_{L}} := \max\limits_{r \in [1,r_{b}]}\;{ U_{L}(r)}.
\end{align*}
The maximal value $\overline{U_{L}}$ defines a global potential barrier for which a particle located at $r \in (1,r_{b})$ with velocity $v_{r}$ and $v_{\theta} = \frac{L}{r}$  such that $\frac{v_r^2}{2} + U_{L}(r) < \overline{U_{L}}$ cannot cross a point $a$ such that $U_{L}(a) = \overline{U}_{L}.$ 
Indeed, arguing by contradiction,  one would have by conservation of the total energy $\frac{v_r^2}{2} + U_{L}(r) =  \frac{v_{a}^2}{2} + U_{L}(a) $ for some $v_{a} \in \RR$ and thus $\frac{v_r^2}{2} + U_{L}(r) \geq \overline{U_{L}}$ which is a contradiction. 
Since we cannot make any assumption on the monotonicity of the function $U_{L}$, it may have many oscillations. 
In such a case, there exist several local potential barriers which yield the existence of trapping sets for the particles as sketched in figure \ref{fig:sketch-phase-space}. 
To construct a solution, we shall thus carefully decompose the phase space $(r,v_{r})$ for each $L \in \RR.$ 
Namely, we shall distinguish between characteristics that intersect the boundaries from those who do not and correspond to trapping sets (see for example \cite{Lucquin} for a definition of a trapping set).
An illustration of the phase space $(r,v_{r})$ corresponding to an effective potential $U_{L}$ having several extrema is given in figure \ref{fig:sketch-phase-space}.
\begin{figure}[h!]
\center
 \scalebox{0.8}{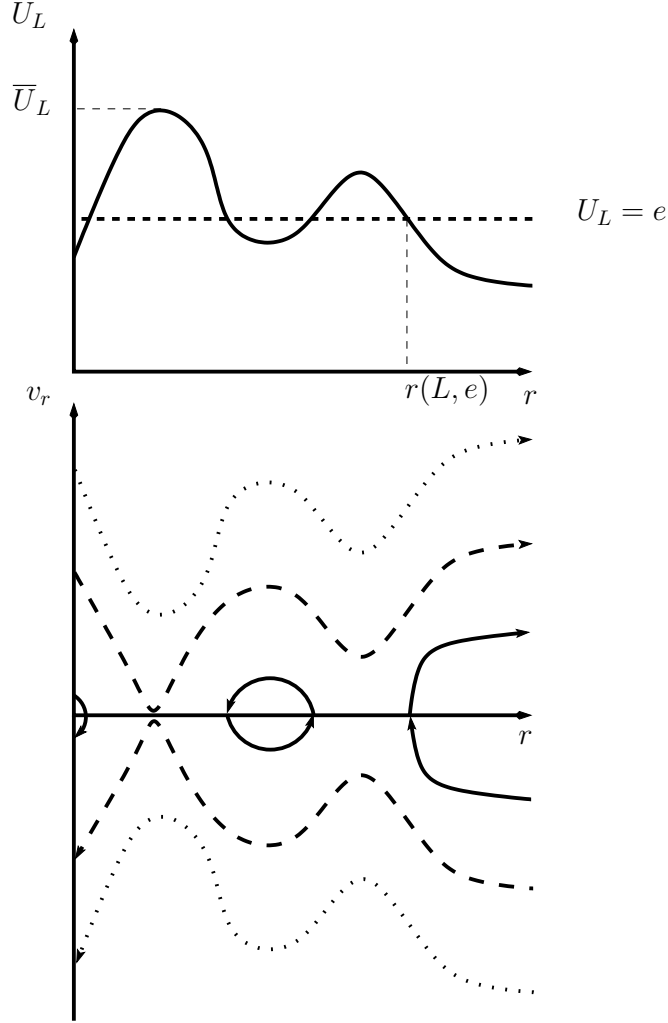 } 
 \caption{Schematic $(r,v_{r})$ phase space decomposition corresponding to an effective potential $U_{L}$. Dotted lines correspond to trajectories of energy level greater than $\overline{U}_{L}$. The dashed line corresponds to a separatrix curve of equation $\frac{v_{r}^2}{2} +U_{L}(r) = \overline{U}_{L}.$ The solid lines correspond to trajectories of energy level lower than $\overline{U}_{L}.$ }\label{fig:sketch-phase-space}
\end{figure}

\subsection*{Characteristics that originate from $r = r_{b}$}
Of particular interest, are those characteristics that originate from the boundary $r = r_{b}$ because they correspond to trajectories of particles coming from the plasma. One has two cases:
\begin{itemize}
\item \textbf{\mathversion{bold} Characteristics with energy level $e > \overline{U_{L}}.$} 
A point of the phase space $(r,v_{r})$ such that $e = \frac{v_{r}^{2}}{2} + U_{L}(r) > \overline{U_{L}}$ is on a characteristic that crosses $r = r_{b}.$ Especially, if $v_{r} < -\sqrt{2(\overline{U_{L}} - U_{L}(r))}$ there is a unique characteristic curve passing through $(r,v_{r})$ that originates from $r_{b}$ with a negative velocity $v_{b} = -\sqrt{v_{r}^{2} + 2(U_{L}(r)-U_{L}(r_{b}))}.$
\item  \textbf{\mathversion{bold} Characteristics with energy level $ e \in [U_{L}(r_{b}),\overline{U_{L}}].$}
If $U_{L}$ has several local maximum, the level curves of equation $\frac{v_{r}^{2}}{2} + U_{L}(r) = e$ may be associated with either closed characteristics or characteristics that originate from $r = r_{b}.$ To distinguish between them, we consider the number
\begin{align}\label{def:r_i}
    r_{i}(L,e) := \min \lbrace a \in [1,r_{b}] \: : \: U_{L}(s) \leq e \:, \forall s \in [a,r_{b}] \rbrace.
\end{align}
By continuity of the function $U_{L}$ this number is well defined and the interval $[r_{i}(L,e),r_{b}]$ is the largest interval containing the point $r_{b}$ in which $U_{L}$ is below the energy level $e \in [ U_{L}(r_{b}), \overline{U_{L}}]$ . If $(r,v_{r})$ is such that $\frac{v_{r}^2}{2} + U_{L}(r) = e \in [ U_{L}(r_{b}), \overline{U_{L}}]$ there is a unique characteristic curve passing through $(r,v_{r})$ originates from $r_{b}$ with a negative velocity $v_{b}= -\sqrt{v_{r}^2 + 2 (U_{L}(r)-U_{L}(r_{b}))}$ if and only if $r > r_{i}(L,e).$
\end{itemize}
The above discussion leads to the following decomposition of the phase space between characteristics that have high energy and characteristics that have low energy:
\begin{align}
&\cD^{b}_i(L) := \cD^{b,1}(L) \cup \cD^{b,2}(L), \label{d_i_b}\\
&    \cD^{b,1}_i(L) = \left \lbrace (r,v_{r}) \in (1,r_{b}) \times \RR \: : \: v_{r} < -\sqrt{2(\overline{U}_{L}-U_{L}(r))} \right \rbrace,\label{d_i_b_1}\\ 
&    \cD^{b,2}_i(L) =  \bigg \lbrace (r,v_{r}) \in (1,r_{b}) \times \RR \: : U_{L}(r_{b}) < \underbrace{\frac{v_{r}^2}{2} + U_{L}(r)}_{=:e} < \overline{U_{L}} \:  , \: r > r_{i}(L,e) \bigg \rbrace. \label{d_i_b_2}
\end{align}
For each point $(r,v_{r}) \in \cD^{b}_i(L)$ there exists a unique characteristics curves that passes through $(r,v_{r})$ and originates from $r = r_{b}$ with a negative velocity $v_{b}= -\sqrt{v_{r}^2 + 2 (U_{L}(r)-U_{L}(r_{b}))}.$

\subsection*{Characteristics that are closed or originates from $r =1$}
Other trajectories are either closed or originate from $r = 1$.
They correspond to point of the phase space $(r,v_{r})$ which are in the complement set of $\cD^{b}_i(L),$ that is
\begin{align*}
    \cD^{pc}_i(L) = \left( (1,r_{b}) \times \RR \right) \setminus \cD^{b}_i(L).
\end{align*}
The function $f_i$ defined by \eqref{def_fi} is taken to be zero on closed characteristics. It could have been any arbitrary function that one may interpret as the trace of some transient solution. Accordingly with \cite{Laframboise}, we assumed these closed characteristics to be unpopulated. From a mathematical point of view, considering the distribution function to be non zero would add some additional terms in the expression of the macroscopic density that we discard for the sake of conciseness of this work.

Then one has for $L \in \RR$ the phase space decomposition
\begin{align*}
(1,r_{b}) \times \RR =  \cD^{b}_i(L) \cup \cD^{pc}_i(L).
\end{align*}
Using this phase space decomposition and the fact that the solutions of the Vlasov equation \eqref{Vlasov-i} are constant on the characteristics, we define
\begin{align}
f_{i}(r,v_r,v_\theta) := \begin{cases}
0 \text{ if } (r,v_r) \in  \cD^{pc}_i(L) \text{ with } L = r v_{\theta}, \\
f_{i}^{b}\left( -\sqrt{v_{r}^2 + 2(U_{L}(r)- U_{L}(r_{b}))}; \frac{rv_{\theta}}{r_{b}}\right) \text{ if } (r,v_{r}) \in \cD^{b}_i(L) \text{ with } L = r v_{\theta}.
\end{cases} \label{def_fi}
\end{align}
In view of the above construction, one has the following:
\begin{proposition}\label{prop:weak_sol} Consider $f_{i}^{b}:\RR^{-}_{*} \times \RR \rightarrow \RR^{+}$ a distribution of velocities for incoming positively charged particles (ions) that is essentially bounded. Therefore $f_{i}$ defined by \eqref{def_fi} is a weak solution of the Vlasov equation in the weak sense given by Definition \ref{def_weak_sol_vlasov}.
\begin{proof}
See the appendix \ref{appendix}.
\end{proof}
\end{proposition}
One can express the macroscopic density explicitly in terms of the effective potential $U_{L}$. This will be of great help for the analysis.
\begin{proposition}\label{prop:n_i formula} Consider $f_{i}^{b}:\RR^{-}_{*} \times \RR\to \RR^{+}$ a distribution of velocities for incoming positively charged particles (ions). With $f_{i}$ defined by \eqref{def_fi} the macroscopic density is given by
\begin{align}
    &rn_{i}(r) = \displaystyle  \int_{-\infty}^{+\infty} \int_{-\infty}^{-\sqrt{2 (\overline{U}_{L} -U_{L}(r_{b}))}} \frac{\vert w_{r} \vert }{\sqrt{w_{r}^2 -2(U_{L}(r)-U_{L}(r_{b}))}} f_{i}^{b}\left( w_{r} , \frac{L}{r_{b}} \right) \,dw_r\,dL \nonumber \\
    &+2 \int_{-\infty}^{+\infty} \mathbbm{1}_{\lbrace U_{L}(r_{b}) - U_{L}(r) < 0 \rbrace} \int_{ \cW_{i,1}^{-}(r,L) } \frac{ \vert w_{r} \vert }{\sqrt{w_{r}^2-2(U_{L}(r)-U_{L}(r_{b}))}} f_{i}^{b} \left( w_{r}; \frac{L}{r_{b}} \right)\,dw_r\,dL\nonumber \\
    &+2 \int_{-\infty}^{+\infty} \mathbbm{1}_{\lbrace U_{L}(r_{b}) - U_{L}(r) \geq 0 \rbrace} \int_{\cW_{i,2}^{-}(r,L)} \frac{ \vert w_{r} \vert }{\sqrt{w_{r}^2-2(U_{L}(r)-U_{L}(r_{b}))}} f_{i}^{b} \left( w_{r}; \frac{L}{r_{b}} \right)\,dw_r\,dL \label{ni}
\end{align}
where
\begin{align*}
    &\cW_{i,1}^{-}(r,L):= \bigg\lbrace w_{r} \in \RR \: : \:  - \sqrt{2 (\overline{U}_{L}-U_{L}(r_{b}))} < w_{r} < -\sqrt{2(U_{L}(r)-U_{L}(r_{b}))} \: \\&\qquad\qquad\qquad\qquad\qquad\qquad\qquad\qquad\qquad\qquad\qquad\qquad\qquad\text{and}\quad r > r_{i}\left( L, \frac{w_{r}^{2}}{2} + U_{L}(r_{b}) \right) \bigg \rbrace,\\
    &\cW_{i,2}^{-}(r,L) := \left \lbrace w_{r} \in \RR \: : \:  - \sqrt{2 (\overline{U}_{L}-U_{L}(r_{b}))}< w_{r} < 0   \: \quad\text{and}\quad r > r_{i}\left( L, \frac{w_{r}^{2}}{2} + U_{L}(r_{b}) \right) \right \rbrace.
\end{align*}
and the radial current density is given by:
\begin{align}
j_{i}(r) = \frac{1}{r} \int_{L = -\infty}^{L = +\infty} \int_{-\infty}^{- \sqrt{2 (\overline{U}_{L}- U_{L}(r_{b}))}} f_{i}^{b}\left( w_{r}; \frac{L}{r_{b}} \right)\, w_{r}\,dw_{r}\, dL. \label{ji}
\end{align}
\end{proposition}
Note that we only integrate non-negative quantities so that the manipulated integrals are always well-defined (finite or not).
Assumptions on the distribution $f_i^b$ that make $rn_{i}(r)$ be a finite quantity are discussed in the next section.

\begin{proof}
Let $r \in (1,r_b)$. One has by definition and using Fubini-Tonelli theorem
\begin{align*}
  n_{i}(r) := \int_{\RR^2} f_{i}(r,v_{r},v_{\theta}) \,dv_r\, dv_{\theta} =  \int_{\RR}\int_{\RR} f_{i}\left( r,v_{r}, v_{\theta} \right)\, dv_{\theta}\, dv_{r}.
\end{align*}
Using the change of variable $L = r v_{\theta}$, one has
\begin{align*}
 rn_{i}(r) =  \int_{\RR}\int_{\RR} f_{i}\left( r,v_{r}, \frac{L}{r} \right)\, dL\, dv_{r} = \int_{\RR}\int_{\RR} f_{i}\left( r,v_{r}, \frac{L}{r} \right)\, d v_{r}\, dL. 
\end{align*}
In view of the definition of $f_{i}$ at~\eqref{def_fi}, the macroscopic density only integrates on the two sets
\begin{align*}
&\cD_i^{b,1}(r) := \left \lbrace (v_{r},L) \in \RR^2 \: : \:  v_{r} < -\sqrt{2( \overline{U}_{L} -U_{L}(r))} \right \rbrace,\\
&\cD_i^{b,2}(r) := \bigg \lbrace (v_{r},L) \in \RR^2 \: : \:    U_{L}(r_{b}) - U_{L}(r) < \frac{v_{r}^2}{2} < \overline{U}_{L}-U_{L}(r) \\&\qquad\qquad\qquad\qquad\qquad\qquad\qquad\qquad\qquad\qquad\quad\text{and}\quad \:  r > r_{i}\left( L, \frac{v_{r}^{2}}{2} + U_{L}(r) \right)  \bigg \rbrace
\end{align*}
Using the definition of $f_{i},$ one therefore obtains 
\begin{align*}
   & r n_{i}(r)  = \int_{-\infty}^{ +\infty} \int_{-\infty}^{-\sqrt{2(\overline{U}_{L}-U_{L}(r))}} f_{i}^{b} \left( -\sqrt{v_{r}^2 + 2(U_{L}(r) - U_{L}(r_{b}))}; \frac{L}{r_{b}} \right)\,dv_{r}\, dL\\
   & +\int_{\cD_i^{b,2}(r)} f_{i}^{b} \left( -\sqrt{v_{r}^2 + 2(U_{L}(r) - U_{L}(r_{b}))}; \frac{L}{r_{b}} \right)\,dv_{r}\, dL
\end{align*}
For the foregoing computation, one sets
\begin{align*}
&I_{1} := \int_{-\infty}^{ +\infty} \int_{-\infty}^{-\sqrt{2(\overline{U}_{L}-U_{L}(r))}} f_{i}^{b} \left( -\sqrt{v_{r}^2 + 2(U_{L}(r) - U_{L}(r_{b}))}; \frac{L}{r_{b}} \right)\,dv_{r}\, dL\\
&I_{2} :=\int_{\cD_i^{b,2}(r)} f_{i}^{b} \left( -\sqrt{v_{r}^2 + 2(U_{L}(r) - U_{L}(r_{b}))}; \frac{L}{r_{b}} \right)\,dv_{r}\, dL. \\
\end{align*}
For the first integral $I_{1}$, one uses the change of variable $w_{r} = -\sqrt{v_{r}^2 + 2(U_{L}(r) - U_{L}(r_{b}))}$ so that one gets
\begin{align*}
   & I_{1} =  \displaystyle  \int_{-\infty}^{+\infty} \int_{-\infty}^{-\sqrt{2 (\overline{U}_{L} -U_{L}(r_{b}))}} \frac{\vert w_{r} \vert }{\sqrt{w_{r}^2 -2(U_{L}(r)-U_{L}(r_{b}))}} f_{i}^{b}\left( w_r , \frac{L}{r_{b}} \right)\, dw_{r}\, dL.
\end{align*}
Regarding the definition of the set $\cD_i^{b,2}(r)$, one splits it in two parts according to the sign of $U_{L}(r_{b})-U_{L}(r).$ Consider for $L \in \RR$ being fixed, the two sets of radial velocities
\begin{align*}
    &\cV_{i,1}(r,L) := \left \lbrace v_{r} \in \RR \: : \:  \vert v_{r} \vert < \sqrt{2 (\overline{U}_{L}-U_{L}(r))} \: \quad\text{and}\quad r > r_{i}\left( L, \frac{v_{r}^{2}}{2} + U_{L}(r) \right) \right \rbrace,\\
    &\cV_{i,2}(r,L) := \bigg \lbrace v_{r} \in \RR \: : \:  \sqrt{2 (U_{L}(r_{b}) - U_{L}(r))} < \vert v_{r} \vert < \sqrt{2 (\overline{U}_{L}-U_{L}(r))} \: \\&\qquad\qquad\qquad\qquad\qquad\qquad\qquad\qquad\qquad\qquad\qquad\qquad\quad\text{and}\quad r > r_{i}\left( L, \frac{v_{r}^{2}}{2} + U_{L}(r) \right) \bigg \rbrace.
\end{align*}
One therefore splits the second integral $I_{2}$ into $I_{2} = I_{2,1} + I_{2,2}$ with
\begin{align*}
&I_{2,1} := \int_{-\infty}^{+\infty} \mathbbm{1}_{\lbrace U_{L}(r_{b}) - U_{L}(r) < 0 \rbrace} \int_{ \cV_{i,1}(r,L) } f_{i}^{b} \left( -\sqrt{v_{r}^2 + 2(U_{L}(r) - U_{L}(r_{b}))}; \frac{L}{r_{b}} \right)\,dv_{r}\, dL,\\
&I_{2,2} := \int_{-\infty}^{+\infty} \mathbbm{1}_{\lbrace U_{L}(r_{b}) - U_{L}(r) \geq 0 \rbrace} \int_{ \cV_{i,2}(r,L) } f_{i}^{b} \left( -\sqrt{v_{r}^2 + 2(U_{L}(r) - U_{L}(r_{b}))}; \frac{L}{r_{b}} \right)\,dv_{r}\, dL.
\end{align*}
One now computes the first integral $I_{2,1}.$ One remarks that the set $\cV_{1}(r,L)$ is symmetric with respect to $v_{r} = 0$ and that the integrand also is. By symmetry one therefore has 
\begin{align*}
    I_{2,1} = 2\int_{-\infty}^{+\infty} \mathbbm{1}_{\lbrace U_{L}(r_{b}) - U_{L}(r) < 0 \rbrace} \int_{ \cV_{i,1}^{-}(r,L) } f_{i}^{b} \left( -\sqrt{v_{r}^2 + 2(U_{L}(r) - U_{L}(r_{b}))}; \frac{L}{r_{b}} \right)\,dv_r\,dL\\
\end{align*}
where one now considers only the negative velocities: 
\begin{align*}
    \cV_{i,1}^{-}(r,L):= \left \lbrace v_{r} \in \RR \: : \:  - \sqrt{2 (\overline{U}_{L}-U_{L}(r))} < v_{r} < 0 \: \quad\text{and}\quad r > r_{i}\left( L, \frac{v_{r}^{2}}{2} + U_{L}(r) \right) \right \rbrace.
\end{align*}
Using the change of  variable $w_{r} = -\sqrt{v_{r}^2 + 2(U_{L}(r)-U_{L}(r_{b}))}$, one gets
\begin{align*}
    &I_{2,1} = 2 \int_{-\infty}^{+\infty} \mathbbm{1}_{\lbrace U_{L}(r_{b}) - U_{L}(r) < 0 \rbrace} \int_{ \cW_{i,1}^{-}(r,L) } \frac{ \vert w_{r} \vert }{\sqrt{w_{r}^2-2(U_{L}(r)-U_{L}(r_{b}))}} f_{i}^{b} \left( w_{r}; \frac{L}{r_{b}} \right)\,dw_r\,dL\\
    \end{align*}
where the set $\cW_{i,1}^{-}(r,L)$ is the image of $\cV_{i,1}^{-}(r,L)$ by the change of variable $v_{r} \mapsto w_{r}$, namely:
    \begin{align*}
    &\cW_{i,1}^{-}(r,L):=\bigg\lbrace w_{r} \in \RR \: : \:  - \sqrt{2 (\overline{U}_{L}-U_{L}(r_{b}))} < w_{r} < -\sqrt{2(U_{L}(r)-U_{L}(r_{b}))} \: \\&\qquad\qquad\qquad\qquad\qquad\qquad\qquad\qquad\qquad\qquad\qquad\qquad\text{and}\quad r > r_{i}\left( L, \frac{w_{r}^{2}}{2} + U_{L}(r_{b}) \right) \bigg \rbrace.
\end{align*}
One now computes the second integral $I_{2,2}.$
The set $\cV_{i,2}(r,L)$ is decomposed as $\cV_{i,2}(r,L) := \cV_{i,2}^{+}(r,L) \cup \cV_{i,2}^{-}(r,L)$
where
\begin{align*}
    &\cV_{i,2}^{+}(r,L) = \bigg\lbrace v_{r} \in \RR \: : \:  \sqrt{2 (U_{L}(r_{b}) - U_{L}(r)} < v_{r}< \sqrt{2 (\overline{U}_{L}-U_{L}(r))} \: \\&\qquad\qquad\qquad\qquad\qquad\qquad\qquad\qquad\qquad\qquad\qquad\qquad\text{and}\quad r > r_{i}\left( L, \frac{v_{r}^{2}}{2} + U_{L}(r) \right) \bigg \rbrace,\\
    &\cV_{i,2}^{-}(r,L) = \bigg\lbrace v_{r} \in \RR \: : \:  - \sqrt{2 (\overline{U}_{L}-U_{L}(r))}< v_{r} < -  \sqrt{2 (U_{L}(r_{b}) - U_{L}(r))}   \: \\&\qquad\qquad\qquad\qquad\qquad\qquad\qquad\qquad\qquad\qquad\qquad\qquad\text{and}\quad r > r_{i}\left( L, \frac{v_{r}^{2}}{2} + U_{L}(r) \right) \bigg\rbrace.
\end{align*}
This yields the following splitting of the integral $I_{2,2},$
\begin{align*}
    &I_{2,2} = \int_{-\infty}^{+\infty} \mathbbm{1}_{\lbrace U_{L}(r_{b}) - U_{L}(r) \geq 0 \rbrace} \int_{ \cV_{i,2}^{+}(r,L) } f_{i}^{b} \left( -\sqrt{v_{r}^2 + 2(U_{L}(r) - U_{L}(r_{b}))}; \frac{L}{r_{b}} \right)\,dv_r\,dL\\
    &+\int_{-\infty}^{+\infty} \mathbbm{1}_{\lbrace U_{L}(r_{b}) - U_{L}(r) \geq 0 \rbrace} \int_{ \cV_{i,2}^{-}(r,L) } f_{i}^{b} \left( -\sqrt{v_{r}^2 + 2(U_{L}(r) - U_{L}(r_{b}))}; \frac{L}{r_{b}} \right)\,dv_r\,dL.
\end{align*}
For each integral, one uses again the change of variable $w_{r} = -\sqrt{v_{r}^2 + 2(U_{L}(r)-U_{L}(r_{b}))}$ so that one eventually obtains 
\begin{align*}
    I_{2,2} = 2 \int_{-\infty}^{+\infty} \mathbbm{1}_{\lbrace U_{L}(r_{b}) - U_{L}(r) \geq 0 \rbrace} \int_{\cW_{i,2}^{-}(r,L)} \frac{ \vert w_{r} \vert }{\sqrt{w_{r}^2-2(U_{L}(r)-U_{L}(r_{b}))}} f_{i}^{b} \left( w_{r}; \frac{L}{r_{b}} \right)\,dw_r\,dL
\end{align*}
where the set $\cW_{i,2}^{-}(r,L)$ is the image of the set $\cV_{i,2}^{-}(r,L)$ by the change of variable $v_{r} \mapsto w_{r},$ namely:
\begin{align*}
&\cW_{i,2}^{-}(r,L) = \left \lbrace w_{r} \in \RR \: : \:  - \sqrt{2 (\overline{U}_{L}-U_{L}(r_{b}))}< w_{r} < 0   \: \quad\text{and}\quad r > r_{i}\left( L, \frac{w_{r}^{2}}{2} + U_{L}(r_{b}) \right) \right \rbrace.
\end{align*}

Gathering all the integrals together yields the expression of the macroscopic density \eqref{ni}.
For the current density, using a similar decomposition of the integral and symmetry arguments one is led to the expression \eqref{ji}.
\end{proof}

\begin{remark} In the expression of the macroscopic density \eqref{ni}, the first integral correspond the density carried by characteristics that travel from $r = r_{b}$ to $r = 1.$ These characteristics also carry some current density. The other integrals correspond to a density carried by characteristics that start from $r= r_{b}$ and go back to $r=r_{b}$ because they correspond to low energy levels. Particles on these characteristics do not have enough energy to overcome the global potential barrier $\overline{U_{L}}.$ On these characteristics there is no current. This eventually explains why the current density \eqref{ji} has only one contribution.
\end{remark}

\subsection{Eletronic phase diagram}
Concerning the electronic phase diagram, the reasoning is similar as for the ionic phase diagram except that, since the electronic charge is now negative, the particles interact with the external electric field with an opposite sign.
In other words, $d\phi/dr$ is replaced by $-d\phi/dr$.
We make use of this analogy to simplify the presentation of the electronic phase diagram.

The characteristics associated with the Vlasov equation \eqref{Vlasov-e} are the solutions to the ordinary differential equations
\begin{equation} \label{char-e}
\begin{cases}
\displaystyle\frac{d}{dt}r(t) = v_{r}(t),\\
\displaystyle\frac{d}{dt}v_{r}(t) = \frac{v_{\theta}(t)^2}{r(t)} + \frac{d\phi}{dr}(r(t)),\\
\displaystyle\frac{d}{dt}v_{\theta}(t) = \frac{-v_{r}(t) v_{\theta}(t)}{r(t)}.
\end{cases}
\end{equation}
Since $d\phi/dr$ is Lipschitz continuous, for each initial condition $(r_{0},v_{r,0},v_{\theta,0}) \in (1,r_{b}) \times \RR^2$ there exists a unique solution $(r,v_r,v_\theta) \in C^{1}( \left(t_{\textnormal{inc}}(r_0,v_{r,0},v_{\theta,0}), t_{\textnormal{out}}(r_0,v_{r,0},v_{\theta,0})\right); $ $[1,r_{b}] \times \RR^2)$ to Equation~\eqref{char-e}, where 
\begin{align*}
&t_{\textnormal{inc}}(r_0,v_{r,0},v_{\theta,0})= \inf \lbrace t' \leq 0 \: : \: r(t) \in (1,r_{b})\: \:  \forall t \in (t',0) \rbrace, \\ &t_{\textnormal{out}}(r_0,v_{r,0},v_{\theta,0})= \sup  \lbrace t' \geq 0 \: : \: r(t) \in (1,r_{b}) \: \:  \forall t \in (0,t') \rbrace
\end{align*}
denote respectively the incoming time and the outgoing time of the characteristics in the interval $(1,r_{b})$. 
They are finite or infinite.
One has two constants of motion:  the total energy and the angular momentum. 
Indeed, the characteristics satisfy for all $t \in (t_{\textnormal{inc}}(r_0,v_{r,0},v_{\theta,0}), $ $ t_{\textnormal{out}}(r_0,v_{r,0},v_{\theta,0})),$
\begin{align*}
&\frac{d}{dt} \left(  \frac{v_{r}^{2}(t)+ v_{\theta}^2(t)}{2} - \phi(r(t)) \right)= 0, \\
&\frac{d}{dt} \left( r(t)v_{\theta}(t) \right) = 0.
\end{align*}
Therefore the characteristics are contained in the following level sets defined for $L \in \RR$ and $e \in \RR$ by
\begin{align*}
\mathcal{C}_{L,e} := \Big\lbrace (r,v_{r},v_{\theta}) \in (1,r_{b}) \times \RR^2 \: : \:  r v_{\theta} = L \text{ and } \frac{v_r^{2} + v_{\theta}^2}{2} - \phi(r) = e \Big\rbrace.
\end{align*}
These sets are used to describe the phase space according to the values of $L$ and $e.$ In this regard, it is convenient to introduce for $L \in \RR$ the effective potential defined by
\begin{align*}
\forall r \in [1,r_{b}], \qquad V_{L}(r) = \frac{L^2}{2r^2} - \phi(r). 
\end{align*}
The continuity of $V_{L}$ follows from the continuity of $\phi$. The function $V_{L}$ therefore reaches reaches its maximum at some point in $[1,r_{b}]$. It maximum value is denoted
\begin{align*}
     \overline{V_L} := \underset{ r \in [1,r_{b}]} \max V_{L}(r)
\end{align*}
Similarly as for the ions, it defines a potential barrier and without further monotony assumption on $V_{L}$, it may exists several local potential barrier. To construct a weak solution, we shall thus carefully decompose the phase $(r,v_{r})$ for each $L \in \RR.$ Namely, we shall distinguish between characteristics that intersect the boundaries from those who do not and correspond to trapping sets. The construction is analogous to the previous one for the ions. We refer the reader to the previous section for the details.
We define
\begin{align}\label{def:r_e}
    r_{e}(L,e) := \min \lbrace a \in [1,r_{b}] \: : \: V_{L}(s) \leq e \:, \forall s \in [a,r_{b}] \rbrace.
\end{align}
and consider the following sets
\begin{align*}
&\cD^{b}_e(L) := \cD^{b,1}(L) \cup \cD^{b,2}(L),\\
&    \cD^{b,1}_e(L) := \left \lbrace (r,v_{r}) \in (1,r_{b}) \times \RR \: : \: v_{r} < -\sqrt{2(\overline{V_L}-V_{L}(r))} \right \rbrace,\\
&    \cD^{b,2}_e(L) :=  \bigg\lbrace (r,v_{r}) \in (1,r_{b}) \times \RR \: : V_{L}(r_{b}) < \underbrace{\frac{v_{r}^2}{2} + V_{L}(r)}_{=:e} < \overline{V_{L}} \:  , \: r > r_{e}(L,e) \bigg \rbrace,\\
&\cD^{pc}_e(L) := (0,1) \times \RR \setminus \cD_e^{b}(L).
\end{align*}
One has the following decomposition:
\begin{align*}
(1,r_{b}) \times \RR = \cD_e^{pc}(L) \cup \cD_e^{b}(L) 
\end{align*}
The domain $\cD^{b}_e(L)$ corresponds to characteristics that originate from the boundary $r = r_{b}.$
The domain $\cD^{pc}_e(L)$ corresponds to characteristics curves that either originates from the probe or are closed and do not intersect the boundaries. 
Using this phase space decomposition and the fact that the solutions of the Vlasov equation \eqref{Vlasov-i} are constant on the characteristics, we define
\begin{align}
f_{e}(r,v_r,v_\theta) := \begin{cases}
0 \text{ if } (r,v_r) \in \cD_e^{pc}(L) \text{ with } L = r v_{\theta}, \\
f_{e}^{b}\left( -\sqrt{v_{r}^2 + 2(V_{L}(r)- V_{L}(r_{b}))}; \frac{rv_{\theta}}{r_{b}}\right) \text{ if } (r,v_{r}) \in \cD_e^{b}(L) \text{ with } L = r v_{\theta}.
\end{cases} \label{def_fe}
\end{align}
Following the same reasoning as for the ions one has,

\begin{proposition} Consider $f_{e}^{b}:\RR^{-}_{*} \times \RR \rightarrow \RR^{+}$ a distribution of velocities for incoming negatively charged particles (electrons) that is essentially bounded. The function $f_{e}$ defined by \eqref{def_fe} is a weak solution of the Vlasov equation in the sense of Definition \ref{def_weak_sol_vlasov}.
\end{proposition}

\begin{proposition}\label{prop:n_e formula} Consider $f_{e}^{b}:\RR^{-}_{*} \times \RR\to \RR^{+}$ a distribution of velocities for incoming negatively charged particles (electrons). With $f_{e}$ defined by \eqref{def_fe} the macroscopic density is given by
\begin{align}
    &rn_{e}(r) = \displaystyle  \int_{-\infty}^{+\infty} \int_{-\infty}^{-\sqrt{2 (\overline{V_L} -V_{L}(r_{b}))}} \frac{\vert w_{r} \vert }{\sqrt{w_{r}^2 -2(V_{L}(r)-V_{L}(r_{b}))}} f_{e}^{b}\left( w , \frac{L}{r_{b}} \right) \,dw_r\,dL \nonumber \\
    &+2 \int_{-\infty}^{+\infty} \mathbbm{1}_{\lbrace V_{L}(r_{b}) - V_{L}(r) < 0 \rbrace} \int_{ \cW_{e,1}^{-}(r,L) } \frac{ \vert w_{r} \vert }{\sqrt{w_{r}^2-2(V_{L}(r)-V_{L}(r_{b}))}} f_{e}^{b} \left( w_{r}; \frac{L}{r_{b}} \right)\,dw_r\,dL\nonumber \\
    &+2 \int_{-\infty}^{+\infty} \mathbbm{1}_{\lbrace V_{L}(r_{b}) - V_{L}(r) \geq 0 \rbrace} \int_{\cW_{e,2}^{-}(r,L)} \frac{ \vert w_{r} \vert }{\sqrt{w_{r}^2-2(V_{L}(r)-V_{L}(r_{b}))}} f_{e}^{b} \left( w_{r}; \frac{L}{r_{b}} \right)\,dw_r\,dL \label{ne}
\end{align}
where
\begin{align*}
    &\cW_{e,1}^{-}(r,L):=:= \bigg \lbrace w_{r} \in \RR \: : \:  - \sqrt{2 (\overline{V_L}-V_{L}(r_{b}))} < w_{r} < -\sqrt{2(V_{L}(r)-V_{L}(r_{b}))} \: \\&\qquad\qquad\qquad\qquad\qquad\qquad\qquad\qquad\qquad\qquad\qquad\qquad\text{and}\quad r > r_{e}\left( L, \frac{w_{r}^{2}}{2} + V_{L}(r_{b}) \right) \bigg \rbrace,\\
    &\cW_{e,2}^{-}(r,L) = \bigg\lbrace w_{r} \in \RR \: : \:  - \sqrt{2 (\overline{V_L}-V_{L}(r_{b}))}< w_{r} < 0   \: \quad\text{and}\quad r > r_{e}\left( L, \frac{w_{r}^{2}}{2} + V_{L}(r_{b}) \right) \bigg\rbrace.
\end{align*}
and the radial current density is given by:
\begin{align*}
j_{e}(r) = \frac{1}{r \sqrt{\mu}} \int_{L = -\infty}^{L = +\infty} \int_{-\infty}^{- \sqrt{2 (\overline{V_L}- V_{L}(r_{b}))}} f_{e}^{b}\left( w_{r}; \frac{L}{r_{b}} \right) w_{r}\,dw_r\,dL.
\end{align*}
\end{proposition}

\section{Reformulation of the non linear Poisson equation}
In this section, we consider $f_{e}^{b}:\RR^{-}_{*} \times \RR \rightarrow \RR^{+}$ a distribution of velocities for incoming positively charged particles (ions) and $f_{e}^{b}:\RR^{-}_{*} \times \RR \rightarrow \RR^{+}$ a distribution of velocities for incoming negatively charged particles (electrons). 
It is natural to be interested in hypothesis on these incoming fluxes so that the quantities $n_i(r)$ and $n_e(r)$ defined respectively at~\eqref{ni} and~\eqref{ne} are finite so that their difference make sense.
Nevertheless, we delay this study to the next section. 
We first need to state the Poisson problem associated to the Vlasov equations of the Langmuir probe and give a satisfactory reformulation of the problem.
Recall that we are interested in solutions $\phi \in W^{2,\infty}(1,r_{b})$ to:
\begin{equation}\label{eq:Poisson studied}
\begin{cases}
\displaystyle-\frac{d}{dr}\bigg(r \,\frac{d\phi}{dr}\bigg)(r) = r (n_{i}-n_{e})(r),\\
\phi(1) = \phi_{p} \quad \phi(r_{b}) = 0,
\end{cases}
\end{equation}
where $n_{i}$ is given by \eqref{ni} (Proposition~\ref{prop:n_i formula}) and $n_{e}$ is given by \eqref{ne} (Proposition~\ref{prop:n_e formula}). The main difficulty to obtain existence of solutions lays in the presence of non-local terms in the definition of the right-hand side of~\eqref{eq:Poisson studied}. The idea is to reformulate the problem and to replace the non-local terms by parameters. In the next section, we prove a general existence result whatever value the parameters have. Secondly, we make a good choice for these parameters so that we get back to the original equation.

To ease the reading, the variable of integration $w_r$ will now be simply noted $w$ since it is now understood that we fully concentrate on the radial behavior.

\subsection{Reformulation of the problem}

\subsubsection{A first reformulation}

To deal with the problem raised by the presence of non-local terms (with respect to $\phi$) in the formulation of $n_i$ and $n_e$, we proceed first to a reformulation of the problem. This involves the replacement of the non-locality by parameters that are adjusted later on.
To this purpose, we first define, for any measurable function $\psi$ defined on $[1,r_b]$, the function $\widetilde{\rho}[\psi]:\RR\to[1,r_b]$ by the following formula:
\begin{equation}\label{def:tilde rho}
    \widetilde{\rho}[\psi](e):=\inf \big\{a\in[1,r_b]\,:\,\text{ for a.e } \,s\in[a,r_b],\;\psi(s)\leq e\big\}.
\end{equation}
It is direct from the definitions~\eqref{def:r_i} and~\eqref{def:r_e} to check that
\begin{equation*}
    r_i(L,e)=\widetilde{\rho}[U_L](e),\qquad\text{and}\qquad r_e(L,e)=\widetilde{\rho}[V_L](e).
\end{equation*}
The function $\widetilde{\rho}$ can be understood as a generalization of $r_i(L,e)$ and $r_e(L,e)$. It will be studied for itself later on to make use of its properties.
It is possible to rewrite the quantity $rn_i(r)$ obtained at~\eqref{ni} as follows:
\begin{align}
    &rn_{i}(r) := \displaystyle  \int_{-\infty}^{+\infty} \int_{-\infty}^{-\sqrt{2 (\overline{U_L} -U_{L}(r_{b}))}} \frac{\vert w  \vert }{\sqrt{w ^2 -2(U_{L}(r)-U_{L}(r_{b}))}} f_{i}^{b}\left( w  , \frac{L}{r_{b}} \right) dw\, dL \nonumber \\
    &+2 \int_{-\infty}^{+\infty} \mathbbm{1}_{\lbrace U_{L}(r_{b}) - U_{L}(r) < 0 \rbrace} \int_{-\sqrt{2(\overline{U_L}-U_L(r_b))}}^{ -\sqrt{2(U_L(r)-U_L(r_b))}} \frac{ \vert w  \vert\,f_i^b(w;L/r_b)}{\sqrt{w ^2-2(U_{L}(r)-U_{L}(r_{b}))}}\, \mathbbm{1}_{r \geq \widetilde{\rho}[U_L]\big(\frac{w ^{2}}{2} + U_{L}(r_{b})\big)}\,dw\, dL\nonumber \\
    &+2 \int_{-\infty}^{+\infty} \mathbbm{1}_{\lbrace U_{L}(r_{b}) - U_{L}(r) \geq 0 \rbrace} \int_{-\sqrt{2(\overline{U_L}-U_L(r_b))}}^0 \frac{ \vert w  \vert\,f_i^b(w;L/r_b)}{\sqrt{w ^2-2(U_{L}(r)-U_{L}(r_{b}))}}\, \mathbbm{1}_{r \geq \widetilde{\rho}[U_L]\big(\frac{w ^{2}}{2} + U_{L}(r_{b})\big)}\,dw\, dL.\label{def:n_i}
\end{align}
Similarly, the quantity $rn_e(r)$ obtained at~\eqref{ne} rewrites:
\begin{align}
    &rn_{e}(r) := \displaystyle  \int_{-\infty}^{+\infty} \int_{-\infty}^{-\sqrt{2 (\overline{V_L} -V_{L}(r_{b}))}} \frac{\vert w  \vert }{\sqrt{w ^2 -2(V_{L}(r)-V_{L}(r_{b}))}} f_{i}^{b}\left( w  , \frac{L}{r_{b}} \right) dw\, dL \nonumber \\
    &+2 \int_{-\infty}^{+\infty} \mathbbm{1}_{\lbrace V_{L}(r_{b}) - V_{L}(r) < 0 \rbrace} \int_{-\sqrt{2(\overline{V_L}-V_L(r_b))}}^{ -\sqrt{2(V_L(r)-V_L(r_b))}} \frac{ \vert w  \vert\,f_i^b(w;L/r_b)}{\sqrt{w ^2-2(V_{L}(r)-V_{L}(r_{b}))}}\, \mathbbm{1}_{r \geq \widetilde{\rho}[V_L]\big(\frac{w ^{2}}{2} + V_{L}(r_{b})\big)}\,dw\, dL\nonumber \\
    &+2 \int_{-\infty}^{+\infty} \mathbbm{1}_{\lbrace V_{L}(r_{b}) - V_{L}(r) \geq 0 \rbrace} \int_{-\sqrt{2(\overline{V_L}-V_L(r_b))}}^0 \frac{ \vert w  \vert\,f_i^b(w;L/r_b)}{\sqrt{w ^2-2(V_{L}(r)-V_{L}(r_{b}))}}\, \mathbbm{1}_{r \geq \widetilde{\rho}[V_L]\big(\frac{w ^{2}}{2} + V_{L}(r_{b})\big)}\,dw\, dL.\label{def:n_e}
\end{align}

\subsubsection{The non-linear term}
To have a formulation that is shorter and easier to manipulate, we introduce the function
\begin{equation}\label{def:beta}
\begin{array}{cccc}
     \beta: & \RR\times[1,r_b]\times\RR &\longrightarrow & \RR\vspace{0.1cm}\\
     \,& (\nu,r,L)& \longmapsto & \displaystyle 2\nu+L^2\bigg(\frac{1}{r^2}-\frac{1}{r_b^2}\bigg).
\end{array}
\end{equation}
We now recall the definition of the positive part of a number $x\in\RR$ with is $(x)_+:=\max\{x,0\}$ and the negative part $(x)_-:=\max\{-x,0\}$.
We then use the function $\beta$ above~\eqref{def:beta} to define:
\begin{equation}\label{def:Gamma}
\begin{array}{cccc}
\Gamma: &\RR\times[1,r_b]\times\RR\times\RR& \longrightarrow&\RR\vspace{0.1cm}\\
\,&(\nu,r,w,L)&\longmapsto&\left\{\begin{array}{cl} \displaystyle\frac{(w)_-}{\sqrt{w^2-\beta(\nu,r,L)\,}}&\quad\text{if }w^2>\beta(\nu,r,L),\vspace{0.1cm}\\
0&\quad\text{otherwise.}\end{array}\right.
\end{array}
\end{equation}
Using these definitions, we can rewrite the formulation of $rn_i$ given at~\eqref{def:n_i} in a more compact way as follows:
\begin{equation*}
    rn_i(r)=rn_{i,1}(r)+rn_{i,2}(r)+rn_{i,3}(r)
\end{equation*}
with
\begin{align}
    &rn_{i,1}(r):=\int_{-\infty}^{+\infty}\int_{-\infty}^{-\sqrt{2\overline{U_L}-L^2/r_b^2}}\Gamma\big(\phi(r),r,w,L\big)\,f_i^b\bigg(w,\frac{L}{r_b}\bigg)\,dw\,dL\\\label{tom}
    &rn_{i,2}(r):=2\int_{-\infty}^{+\infty}\mathbbm{1}_{\lbrace \beta(\phi(r),r,L)>0\rbrace}\int_{-\sqrt{2\overline{U_L}-L^2/r_b^2}}^{-\sqrt{\beta(\phi(r),r,L)}}\Gamma\big(\phi(r),r,w,L\big)\,f_i^b\bigg(w,\frac{L}{r_b}\bigg)\,\mathbbm{1}_{r \geq \widetilde{\rho}[U_L]\big(\frac{w ^{2}}{2} + \frac{L^2}{2r_b^2}\big)}\,dw\,dL,\\\label{jerry}
    &rn_{i,3}(r):=2\int_{-\infty}^{+\infty}\mathbbm{1}_{\lbrace \beta(\phi(r),r,L)\leq 0\rbrace}\int_{-\sqrt{2\overline{U_L}-L^2/r_b^2}}^{0}\Gamma\big(\phi(r),r,w,L\big)\,f_i^b\bigg(w,\frac{L}{r_b}\bigg)\,\mathbbm{1}_{r \geq \widetilde{\rho}[U_L]\big(\frac{w ^{2}}{2} + \frac{L^2}{2r_b^2}\big)}\,dw\,dL.
\end{align}
Note that we used $U_L(r_b)=L^2/2r_b^2$ (consequence of $\phi(r_b)=0$). Similarly we can rewrite the formulation of $rn_e$ given at~\eqref{def:n_e} by
\begin{equation*}
    rn_e(r):=rn_{e,1}(r)+rn_{e,2}(r)+rn_{e,3}(r)
\end{equation*}
with
\begin{align*}
    &rn_{e,1}(r):=\int_{-\infty}^{+\infty}\int_{-\infty}^{-\sqrt{2\overline{V_L}-L^2/r_b^2}}\Gamma\big(\!-\phi(r),r,w,L\big)\,f_e^b\bigg(w,\frac{L}{r_b}\bigg)\,dw\,dL\\
    &rn_{e,2}(r):=2\int_{-\infty}^{+\infty}\mathbbm{1}_{\lbrace \beta(\phi(r),r,L)>0\rbrace}\int_{-\sqrt{2\overline{V_L}-L^2/r_b^2}}^{-\sqrt{\beta(\phi(r),r,L)}}\Gamma\big(\!-\phi(r),r,w,L\big)\,f_e^b\bigg(w,\frac{L}{r_b}\bigg)\,\mathbbm{1}_{r \geq \widetilde{\rho}[V_L]\big(\frac{w ^{2}}{2} + \frac{L^2}{2r_b^2}\big)}\,dw\,dL,\\
    &rn_{e,3}(r):=2\int_{-\infty}^{+\infty}\mathbbm{1}_{\lbrace \beta(\phi(r),r,L)\leq 0\rbrace}\int_{-\sqrt{2\overline{V_L}-L^2/r_b^2}}^{0}\Gamma\big(\!-\phi(r),r,w,L\big)\,f_e^b\bigg(w,\frac{L}{r_b}\bigg)\,\mathbbm{1}_{r \geq \widetilde{\rho}[V_L]\big(\frac{w ^{2}}{2} + \frac{L^2}{2r_b^2}\big)}\,dw\,dL.
\end{align*}
Using the positive part function $(\cdot)_+$ allows us to sum the two last terms~\eqref{tom} and~\eqref{jerry} and obtain this more simple formulation:
\begin{equation*}\begin{split}
    &rn_{i,2}(r)+rn_{i,3}(r)=2\int_{-\infty}^{+\infty}\int_{-\sqrt{2\overline{U_L}-L^2/r_b^2}}^{-\sqrt{\beta(\phi(r),r,L)_+}}\Gamma\big(\phi(r),r,w,L\big)\,f_i^b\bigg(w,\frac{L}{r_b}\bigg)\,\mathbbm{1}_{r \geq \widetilde{\rho}[U_L]\big(\frac{w ^{2}}{2} + \frac{L^2}{2r_b^2}\big)}\,dw\,dL\\
    &=2\int_{-\infty}^{+\infty}\int_{-\infty}^{-\sqrt{\beta(\phi(r),r,L)_+}}\Gamma\big(\phi(r),r,w,L\big)\,f_i^b\bigg(w,\frac{L}{r_b}\bigg)\,\mathbbm{1}_{w^2+\frac{L^2}{r_b^2}<2\overline{U_L}}\;\mathbbm{1}_{r \geq \widetilde{\rho}[U_L]\big(\frac{w ^{2}}{2} + \frac{L^2}{2r_b^2}\big)}\,dw\,dL\\
    &=2\int_{-\infty}^{+\infty}\int_{-\infty}^{+\infty}\Gamma\big(\phi(r),r,w,L\big)\,f_i^b\bigg(w,\frac{L}{r_b}\bigg)\,\mathbbm{1}_{w^2+\frac{L^2}{r_b^2}<2\overline{U_L}}\;\mathbbm{1}_{r \geq \widetilde{\rho}[U_L]\big(\frac{w ^{2}}{2} + \frac{L^2}{2r_b^2}\big)}\,dw\,dL,
\end{split}
\end{equation*}
where for the last equality we used the fact that $\Gamma$ is equal to $0$ whenever $w^2\leq\beta(\nu,r,L)$ or $w\geq0$.
Concerning the first term, we write
\begin{equation*}\begin{split}
    &rn_{i,1}(r)=\int_{-\infty}^{+\infty}\int_{-\infty}^{+\infty}\Gamma\big(\phi(r),r,w,L\big)\,f_i^b\bigg(w,\frac{L}{r_b}\bigg)\,\mathbbm{1}_{w^2+\frac{L^2}{r_b^2}<2\overline{U_L}}\,dw\,dL\\
    &=\int_{-\infty}^{+\infty}\int_{-\infty}^{+\infty}\Gamma\big(\phi(r),r,w,L\big)\,f_i^b\bigg(w,\frac{L}{r_b}\bigg)\,\mathbbm{1}_{w^2+\frac{L^2}{r_b^2}<2\overline{U_L}}\;\mathbbm{1}_{r \geq \widetilde{\rho}[U_L]\big(\frac{w ^{2}}{2} + \frac{L^2}{2r_b^2}\big)}\,dw\,dL,
\end{split}
\end{equation*}
where for the last equality we use the following property of $\widetilde{\rho}$:
\begin{equation*}
    \overline{U_L}\leq e\qquad\Longleftrightarrow\qquad \widetilde{\rho}[U_L](e)=1.
\end{equation*}
If we now make the sum of these two terms and use the general property  $\mathbbm{1}_A+\mathbbm{1}_{A^c}=1$, we are led to
\begin{equation}\label{def:n_i new}
     rn_i(r)=\int_{\RR^2}\Gamma\big(\phi(r),r,w,L\big)\,f_i^b\bigg(w,\frac{L}{r_b}\bigg)\bigg(1+\mathbbm{1}_{w^2+\frac{L^2}{r_b^2}<2\overline{U_L}}\bigg)\mathbbm{1}_{r \geq \widetilde{\rho}[U_L]\big(\frac{w ^{2}}{2} + \frac{L^2}{2r_b^2}\big)}\,dw\,dL.   
\end{equation}
Similarly,
\begin{equation}\label{def:n_e new}
    rn_e(r)=\int_{\RR^2}\Gamma\big(\!-\phi(r),r,w,L\big)\,f_e^b\bigg(w,\frac{L}{r_b}\bigg)\bigg(1+\mathbbm{1}_{w^2+\frac{L^2}{r_b^2}<2\overline{V_L}}\bigg)\mathbbm{1}_{r \geq \widetilde{\rho}[V_L]\big(\frac{w ^{2}}{2} + \frac{L^2}{2r_b^2}\big)}\,dw\,dL.
\end{equation}

\subsection{Replacement of the non-locality by parameters}\label{sec:param}
Now that we have a compact formulation of the right-hand side of~\eqref{eq:studied equation}, there remain to prove the existence result. 
Nevertheless, one difficulty arises due to the presence of ``\emph{non-local}'' terms in the equation.
Throughout this article, we say that a given expression depending on $r$ and $\phi:[1,r_b]\to\RR$ is ``\emph{local}'', if at a given point $r\in[1,r_b]$, this expression depends only on $r$, $\phi(r)$ and on the derivatives of $\phi$ evaluated at point $r$ (or any quantity that can be computed knowing $\phi$ only on arbitrarily small neighborhood of point $r$). 
In this case, the ``\emph{non-local}'' terms in~\eqref{eq:Poisson studied} are $\overline{U_L}$, $\overline{V_L}$, $\widetilde{\rho}[U_L](e)$ and $\widetilde{\rho}[V_L](e)$. Indeed, these terms are computed using a $max$ operator which involves to know the value of the function $\phi$ on a full interval.

The strategy is to temporarily get rid of these non-local terms and replace them by parameters.
We then prove a very general result of existence using standard variational techniques.
The parameters are adjusted later in the article in such a way that the initial problem is recovered.

\subsubsection{The max-parameters}
The first parameters that we introduce, called \emph{max-parameters}, are used to remove the dependency of $n_i$ and $n_e$ with respect to $\overline{U_L}$ and $\overline{V_L}$ respectively.
These parameters are noted respectively $\mU_L$ and $\mV_L$ (the Gothic version of the letters $U$ and $V$). 
We are going solve a relaxed problem involving these parameters $\mU_L$ and $\mV_L$ supposed fixed and, later in the proof, we adjust the value of these parameters in such a way that for almost every $L$,
\begin{equation*}
    \overline{U_L}=\mU_L,\qquad\text{and},\qquad\overline{V_L}=\mV_L.
\end{equation*}
It is then natural with such a strategy to define, in the view of~\eqref{def:n_i new},
\begin{align}
    &rn_{i}[\mU](r) := \int_{\RR^2}\Gamma\big(\phi(r),r,w,L\big)\,f_i^b\bigg(w,\frac{L}{r_b}\bigg)\bigg(1+\mathbbm{1}_{w^2+\frac{L^2}{r_b^2}<2\mU_L}\bigg)\mathbbm{1}_{r \geq \widetilde{\rho}[U_L]\big(\frac{w ^{2}}{2} + \frac{L^2}{2r_b^2}\big)}\,dw\,dL.\label{def:n_i param}
\end{align}
We do observe that in the particular case $\mU_L=\overline{U_L}$ (and we prove \emph{a posteriori} that such a case exists), we recover the initial studied quantity: $rn_{i}[\mU_L=\overline{U_L}](r)=rn_i(r)$. We define analogously the quantity $rn_{e}[\mV_L](r)$ from~\eqref{def:n_e new} by replacing $\overline{V_L}$ by $\mV_L$.

\subsubsection{The barrier parameters}
The second terms that are non-local with respect to the function $\phi$ are $r_i(L,e)$ and $r_e(L,e)$ that give the position of the barrier of potential.
Recall that we rewrote these terms using $\widetilde{\rho}$.
We consider now the ``\emph{barrier-parameters}'', noted $\mR_i(w,L)$ and $\mR_e(w,L)$. 
We introduce $rn_{i}[\mU_L, \mR_i](r)$ with the same formula as for~\eqref{def:n_i param} except that the indicator function for the case $r \geq \widetilde{\rho}[U_L]\big(\frac{w ^{2}}{2} + U_{L}(r_{b})\big)$ is replaced by the indicator function associated to $r>\mR_i(w,L)$. The function $(w,L)\mapsto\mR_i(w,L)$ is chosen to be any fixed function (in this sense it is seen as a parameter) and once again, we recover the previous expression in the particular case (proved \emph{a posteriori} to exist) where $\mR_i(w,L)=\widetilde{\rho}[U_L]\big(\frac{w ^{2}}{2} + U_{L}(r_{b})\big)$ for all $r,w,L$. 
An analogous construction gives the definition of $rn_{e}[\mV, \mR_e](r).$ 

\subsection{The semi-linear problem}

\subsubsection{A local equation with parameters}
Now that have replaced all the non-local terms by parameters in~\eqref{eq:Poisson studied}, we are reduced to study the equation:
\begin{equation}\label{eq:studied equation}
    \forall\,r\in[1,r_b],\qquad-\frac{d}{dr}\bigg(r\,\frac{d\phi}{dr}\bigg)(r)=\widetilde{g}\Big(\phi(r),\,r\Big),
\end{equation}
where $\widetilde{g}:\RR\times[1,r_b]\to\RR$ is defined by
\begin{equation}\label{def:tilde g}
    \widetilde{g}(\nu,r):=g_i(\nu,r)-g_e(\nu,r),
\end{equation}
with
\begin{equation}\label{def:g_i}
    g_i(\nu,r):=\int_{\RR^2}\Gamma\big(\nu,r,w,L\big)\,f_i^b\bigg(w,\frac{L}{r_b}\bigg)\bigg(1+\mathbbm{1}_{w^2+\frac{L^2}{r_b^2}<2\mU_L}\bigg)\mathbbm{1}_{r \geq \mR_i(w,L)}\,dw\,dL.
\end{equation}
and
\begin{equation}\label{def:g_e}
        g_e(\nu,r):\int_{\RR^2}\Gamma\big(\!-\!\nu,r,w,L\big)\,f_e^b\bigg(w,\frac{L}{r_b}\bigg)\bigg(1+\mathbbm{1}_{w^2+\frac{L^2}{r_b^2}<2\mV_L}\bigg)\mathbbm{1}_{r \geq \mR_e(w,L)}\,dw\,dL.
\end{equation}
With such a formulation at hand, we can expect to obtain the existence of a solution using standard variational arguments.

\subsubsection{A change of variable}
One last transformation consists in setting, for $x\in[0,1]$,
\begin{equation*}
    \psi(x):=\phi\big((r_b)^x\big)-\phi_p(1-x)
\end{equation*}
so that $\psi(0)=\phi(1)-\phi_p=0$ and $\psi(1)=\phi(r_b)=0$. With the change of variable $r=(r_b)^x$, we get
\begin{equation*}\begin{split}
    -\psi''(x)&=-(r_b)^x\,\log(r_b)^2\Big(\phi'\big((r_b)^x\big)+(r_b)^x\phi''\big((r_b)^x\big)\Big)\\
    &=-(r_b)^x\,\log(r_b)^2\,\frac{d}{dr}\bigg(r\frac{d\phi}{dr}\bigg)(r)
    \end{split}
\end{equation*} 
The studied equation~\eqref{eq:studied equation} is therefore equivalent to
\begin{equation}\label{eq:equivalent equation}
    \forall\,x\in[0,1],\qquad-\frac{d^2\psi}{dx^2}(x)=g\Big(\psi(x),\,x\Big),
\end{equation}
where
\begin{equation}\label{def:g}
g(\nu,x):=(r_b)^x\,\log(r_b)^2\:\widetilde{g}\Big(\nu+\phi_p(1-x),\,(r_b)^x\,\Big).
\end{equation}
It is possible to recover $\phi$ from $\psi$ with the formula
\begin{equation}\label{eq:psi to phi}
    \forall\,r\in[1,r_b],\qquad\phi(r)=\psi\bigg(\frac{\log(r)}{\log(r_b)}\bigg)
    +\phi_p\bigg(1-\frac{\log(r)}{\log(r_b)}\bigg).
\end{equation}
One interest of this last formulation~\eqref{eq:equivalent equation} is that it directly involves the second derivative of $\psi$ (which is easier to manipulate) and the Sobolev space $H^1_0([0,1])$.
This formulation also allows to proceed to qualitative description of the solutions $\psi$ invoking convexity arguments (such a study will be done in forthcoming articles).

\section{Existence of a solution}

\subsection{A priori estimates}

The first main question concerning~\eqref{eq:equivalent equation} is the definition problem for the function $g$ and, which is equivalent, the function $\widetilde{g}$. 
Recall that $\widetilde{g}$ is the difference between $g_i$ defined at~\eqref{def:g_i} and $g_e$ defined at~\eqref{def:g_e}.
It is possible to prove with elementary computations that
\begin{equation*}
        \sup\limits_{\nu\in\RR}\;\sup_{r\in[1,r_b]}\int_{-\infty}^{+\infty}\int_{-\infty}^{+\infty}\big|\Gamma\big(\nu,r,w,L\big)\big|\,dw\,dL\;=+\infty.
\end{equation*}
It is therefore not enough to ask $f_i^b$ and $f_e^b$ to be in $L^\infty$ if one wants the functions $g_i$ and $g_e$ to be finite. 
Similar manipulations gives that assuming moreover $f_i^b$ and $f_e^b$ to be in $L^1$ is not enough and extra integrability assumptions are required.

To start with, we prove the following estimate:

\begin{lemma}[Functions $g_i$ and $g_e$ are finite]\label{lem:well defined g_i g_e}
Let $f:\RR^2\to\RR$ measurable and define the function $\Gamma$ with~\eqref{def:Gamma}. Let $p\in[1,2)$.

Then,
\begin{equation*}\begin{split}
    \sup\limits_{\nu\in\RR}\;\sup_{r\in[1,r_b]}\int_{-\infty}^{+\infty}\int_{-\infty}^{+\infty}\frac{|w|^p}{\big|w^2-L^2\big(\frac{1}{r^2}-\frac{1}{r_b^2}\big)-2\nu\big|^\frac{p}{2}}\,\big|f(w,L)\big|\,dw\,dL\leq\;2\|f\|_{L^1}+\frac{4}{2-p}\|f\|_{L^1_L(L^\infty_w(w\,dw))},
    \end{split}
\end{equation*}
where $L^1_L(L^\infty_w(w\,dw))$ is defined at~\eqref{def:Lebesgue 1}.
\end{lemma}

\begin{proof}
Let $p\in[1,2)$ and let $b\in(0,1/2]$. let $L,\nu\in\RR$ and let $r\in[1,r_b]$. We define the set 
\begin{equation*}
    \cO_{b,r}^{L,\nu}:=\left\{w\in\RR:\left|w^2-L^2\left(\frac{1}{r^2}-\frac{1}{r_b^2}\right)-2\nu\right|\,\leq\,b\,w^2\right\}.
\end{equation*}
By definition of $\cO_{b,r}^{L,\nu}$,
\begin{equation}\label{chat}
    \int_{\RR\setminus\cO_{b,r}^{L,\nu}}\frac{|w|^p}{\big|w^2-L^2\big(\frac{1}{r^2}-\frac{1}{r_b^2}\big)-2\nu\big|^\frac{p}{2}}\,\big|f(w,L)\big|\,dw\;\leq\;\frac{1}{b^\frac{p}{2}}\int_{-\infty}^{+\infty}\big|f(w,L)\big|\,dw.
\end{equation}
On the other hand,
\begin{equation}\label{lapin}\begin{split}
    w\in\cO_{b,r}^{L,\nu}\quad&\Longleftrightarrow\quad (b-1)w^2\leq L^2\big(\frac{1}{r^2}-\frac{1}{r_b^2}\big)+2\nu\leq (b+1)w^2\\
    &\Longleftrightarrow\quad\frac{L^2\big(\frac{1}{r^2}-\frac{1}{r_b^2}\big)+2\nu}{1+b}\leq w^2 \leq \frac{L^2\big(\frac{1}{r^2}-\frac{1}{r_b^2}\big)+2\nu}{1-b}.
    \end{split}
\end{equation}
We see that, for $\lambda$ a positive number,
\begin{equation}\label{hamster}\begin{split}
    \int_{\frac{\lambda}{\sqrt{1+b}}}^\lambda\;\frac{|w|^{p-1}dw}{\big|w^2-\lambda^2\big|^\frac{p}{2}}
    \leq\lambda^{p-1} \int_{\frac{\lambda}{\sqrt{1+b}}}^\lambda\;&\frac{dw}{\big|(\lambda+w)(\lambda-w)\big|^\frac{p}{2}}
    \leq\frac{\lambda^{p-1}}{\big|\lambda+\frac{\lambda}{\sqrt{1+b}}\big|^\frac{p}{2}}\int_{\frac{\lambda}{\sqrt{1+b}}}^\lambda\;\frac{dw}{\big|\lambda-w\big|^\frac{p}{2}}
    \\&=\frac{1}{1-\frac{p}{2}}\frac{\big|1-\frac{1}{\sqrt{1+b}}\big|^{1-\frac{p}{2}}}{\big|1+\frac{1}{\sqrt{1+b}}\big|^{\frac{p}{2}}}\leq\frac{1}{1-\frac{p}{2}}.
    \end{split}
\end{equation}
Similarly,
\begin{equation}\label{cochon dinde}\begin{split}
    \int_\lambda^\frac{\lambda}{\sqrt{1-b}}\;\frac{|w|^{p-1}dw}{\big|w^2-\lambda^2\big|^\frac{p}{2}}
    \leq\frac{\lambda^{p-1}}{\sqrt{1-b\,}^{p-1}}\int_\lambda^\frac{\lambda}{\sqrt{1-b}}\;&\frac{dw}{\big|(\lambda+w)(w-\lambda)\big|^\frac{p}{2}}
    \leq\frac{\lambda^{\frac{p}{2}-1}}{2^\frac{p}{2}\sqrt{1-b\,}^{p-1}}\int_\lambda^\frac{\lambda}{\sqrt{1-b}}\;\frac{dw}{|w-\lambda|^\frac{p}{2}}
    \\&=\frac{1}{2^\frac{p}{2}\big(1-\frac{p}{2}\big)}\frac{\big|1-\sqrt{1-b}\big|^{1-\frac{p}{2}}}{\sqrt{1-b\,}^\frac{p}{2}}\leq\frac{1}{1-\frac{p}{2}},
    \end{split}
\end{equation}
where for the last inequality we used $b\leq1/2$.
We note that~\eqref{lapin} implies that $\cO_{b,r}^{L,\nu}$ is non-empty if and only if $L^2(1/r^2-1/r_b^2)+2\nu\geq 0$.
In this case we can choose $\lambda$ such that $\lambda^2=L^2(1/r^2-1/r_b^2)+2\nu.$
Then the computations~\eqref{hamster} and~\eqref{cochon dinde} imply
\begin{equation}\label{furet}\begin{split}
        &\qquad\int_{\cO_{b,r}^{L,\nu}}\frac{|w|^p}{\big|w^2-L^2\big(\frac{1}{r^2}-\frac{1}{r_b^2}\big)-2\nu\big|^\frac{p}{2}}\,\big|f(w,L)\big|\,dw\\&\leq\;\Big(\sup_{w}|w|\,\big|f(w,L)\big|\Big)        \int_{\cO_{b,r}^{L,\nu}}\frac{|w|^{p-1}}{\big|w^2-L^2\big(\frac{1}{r^2}-\frac{1}{r_b^2}\big)-2\nu\big|^\frac{p}{2}}\,dw\\
        &\leq\;\frac{2}{1-\frac{p}{2}}\Big(\sup_{w}|w|\,\big|f(w,L)\big|\Big).
    \end{split}
\end{equation}
If we now gather~\eqref{chat} and~\eqref{furet} and integrate these two estimates for the variable $L$:
\begin{equation*}\begin{split}
    \int_{-\infty}^{+\infty}\int_{-\infty}^{+\infty}&\frac{|w|^p}{\big|w^2-L^2\big(\frac{1}{r^2}-\frac{1}{r_b^2}\big)-2\nu\big|^\frac{p}{2}}\,\big|f(w,L)\big|\,dw\,dL\\&\leq\frac{1}{b^\frac{p}{2}}\|f\|_{L^1}+\frac{2}{1-\frac{p}{2}}\int_{-\infty}^{+\infty}\Big(\sup_{w}|w|\big|f(w,L)\big|\Big)\,dL.
    \end{split}
\end{equation*}
Plugging this back into~\eqref{chien} concludes the proof (choosing $b=1/2$). \end{proof}

\begin{corollary}
Suppose that the functions $f_i^b$ and $f_e^b$ are in $L^1\cap L^1_L(L^\infty_w(w\,dw))$.
Then, the functions $g_i$ and $g_e$ defined at~\eqref{def:g_i}~\eqref{def:g_e} are well-defined and bounded with a bound that depends only on $\|f^b\|_{L^1}$ and $\|f^b\|_{L^1_L(L^\infty_w(w\,dw))}$. 
\end{corollary}
This implies that $\widetilde{g}=g_i-g_e$ is also well-defined and bounded and so is the function $g$ given at~\eqref{def:g}.
\begin{proof}
The definition of $\Gamma$ at~\eqref{def:Gamma} implies
\begin{equation}\label{chien}
    \int_{-\infty}^{+\infty}\int_{-\infty}^{+\infty}\big|\Gamma\big(\nu,r,w,L\big)\big|\,\big|f(w,L)\big|\,dw\,dL\;\leq\;\int_{-\infty}^{+\infty}\int_{-\infty}^{+\infty}\frac{|w|}{\big|w^2-L^2\big(\frac{1}{r^2}-\frac{1}{r_b^2}\big)-2\nu\big|^\frac{1}{2}}\,\big|f(w,L)\big|\,dw\,dL.
\end{equation}
The fact that $g_i$ and $g_e$ are well-defined, and bounded is then a direct corollary of Lemma~\ref{lem:well defined g_i g_e} with $p=1$.
\end{proof}

Now that the functions $g_e$ and $g_i$ are well-defined, we study their regularity:
\begin{lemma}[Regularity of the function $\widetilde{g}$]\label{lem:Holder regularity}
Suppose that the functions $f_i^b$ and $f_e^b$ are in $L^1(\RR^2)$. Suppose also that there exists $0<\gamma<1$ such that $f_i^b$ and $f_e^b$ belong to $L^1_L(L^\infty_w(w\,dw))\cap L^1_w(L^\infty_L\,;dw/|w|^\gamma)$. Recall the these spaces are defined by the norms~\eqref{def:Lebesgue 1} and~\eqref{def:Lebesgue 2}. Define the functions $g_i$ and $g_e$ with~\eqref{def:g_i}~\eqref{def:g_e}. Then we have for all $\nu,\nu'\in\RR$ such that $|\nu'-\nu|\leq1$ and for all $r\in[1,r_b)$,
 \begin{equation*}
   \big| g_i(\nu',r)-g_i(\nu,r)\big|\,\leq\, \frac{C(r)}{\gamma(1-\gamma)}\Big(1+\|f_i^b\|_{L^1}+\|f_i^b\|_{L^1_L(L^\infty_w(w\,dw))}+\|f_i^b\|_{L^1_w(L^\infty_L\,;dw/|w|^\gamma)}\Big)\,|\nu'-\nu|^\frac{\gamma}{2(\gamma+1)},
 \end{equation*}
 where $C$ is a function of $r$ that blows up as $r\to r_b$. The same estimate holds for the function $g_e$ and then for the function $\widetilde{g}$.
\end{lemma}

\begin{proof}
Let $\nu'<\nu\in\RR$ such that $\nu-\nu'\leq 1$ and let $r\in[1,r_b)$. We consider the number $1<p<2$ such that $\gamma=(p-1)/(3-p)$. We define
\begin{equation*}
    \cP_{\nu,\nu'}^{r,p}:=\bigg\{(w,L)\in\RR^2\;:\;\bigg|w^2-L^2\bigg(\frac{1}{r^2}-\frac{1}{r_b^2}\bigg)-2\nu'\bigg|\geq\frac{\nu-\nu'}{|w|^{2\frac{p-1}{3-p}}}\bigg\}.
\end{equation*}

\emph{Step 1: Regularity property on $\cP_{\nu,\nu'}^{r,p}$.}
By convexity inequality, we have that for all $a>0$ and for all $h\geq0$,
\begin{equation*}
    \frac{1}{\sqrt{a}}-\frac{1}{\sqrt{a+h}}\leq \frac{h}{2\sqrt{a}^3}.
\end{equation*}
Thus,
\begin{equation}\label{aigle}\begin{split}
I_{\nu,\nu'}^{r,p}:=\int_{\cP_{\nu,\nu'}^{r,p}}\Big|\Gamma(\nu',r,w,L)&-\Gamma(\nu,r,w,L)\Big|\,\bigg|f_i^b\bigg(w,\frac{L}{r_b}\bigg)\bigg|\bigg(1+\mathbbm{1}_{w^2+\frac{L^2}{r_b^2}<2\mU_L}\bigg)\mathbbm{1}_{r \geq \mR_i(w,L)}\,dw\,dL
\\&\leq\int_{\cP_{\nu,\nu'}^{r,p}}\frac{|w|\;(\nu-\nu')}{\big|w^2-L^2\big(\frac{1}{r^2}-\frac{1}{r_b^2}\big)-2\nu'\big|^{\frac{3}{2}}}\bigg|f_i^b\bigg(w,\frac{L}{r_b}\bigg)\bigg|\,dw\,dL\\
&\leq\int_{\RR^2}\frac{|w|^p\;(\nu-\nu')^\frac{p-1}{2}}{\big|w^2-L^2\big(\frac{1}{r^2}-\frac{1}{r_b^2}\big)-2\nu'\big|^{\frac{p}{2}}}\bigg|f_i^b\bigg(w,\frac{L}{r_b}\bigg)\bigg|\,dw\,dL,
\end{split}
\end{equation}
where for the last inequality we used the definition of $\cP_{\nu,\nu'}^{r,p}$ since it implies
\begin{equation*}
    \frac{|w|^{1-p}\,(\nu-\nu')^{\frac{3}{2}-\frac{p}{2}}}{\big|w^2-L^2\big(\frac{1}{r^2}-\frac{1}{r_b^2}\big)-2\nu'\big|^{\frac{3}{2}-\frac{p}{2}}}\leq1.
\end{equation*}
We now simply make use of Lemma~\ref{lem:well defined g_i g_e} to obtain that the term studied at~\eqref{aigle} is bounded by \begin{equation}\label{vautour}
    I_{\nu,\nu'}^{r,p}\leq C\frac{\|f_i^b\|_{L^1}+\|f_i^b\|_{L^1_L(L^\infty_w(w\,dw))} }{2-p}\,(\nu-\nu')^\frac{p-1}{2}=C\Big(\|f_i^b\|_{L^1}+\|f_i^b\|_{L^1_L(L^\infty_w(w\,dw))}\Big)\frac{\gamma+1}{\gamma-1}(\nu-\nu')^\frac{\gamma}{\gamma+1}
\end{equation}

\emph{Step 2: Regularity property on $\RR^2\setminus\cP_{\nu,\nu'}^{r,p}$.}
We need first to separate the analysis into $2$ cases. for that purpose, we introduce
\begin{equation*}
    \cN_{\nu,\nu'}^{r}:=\bigg\{(w,L)\in\RR^2:w^2-L^2\bigg(\frac{1}{r^2}-\frac{1}{r_b^2}\bigg)-2\nu'>0\bigg\}.
\end{equation*}
The positiveness of $\Gamma$ gives
\begin{equation}\label{faucon}\begin{split}
\int_{\cN_{\nu,\nu'}^{r}\setminus\cP_{\nu,\nu'}^{r,p}}\Big|\Gamma(\nu',r,w,L)&-\Gamma(\nu,r,w,L)\Big|\,\bigg|f_i^b\bigg(w,\frac{L}{r_b}\bigg)\bigg|\bigg(1+\mathbbm{1}_{w^2+\frac{L^2}{r_b^2}<2\mU_L}\bigg)\mathbbm{1}_{r \geq \mR_i(w,L)}\,dw\,dL
\\&\leq\int_{\cN_{\nu,\nu'}^{r}\setminus\cP_{\nu,\nu'}^{r,p}}\Big|\Gamma(\nu',r,w,L)\Big|\,\bigg|f_i^b\bigg(w,\frac{L}{r_b}\bigg)\bigg|\,dw\,dL.
\end{split}
\end{equation}
On the other hand, outside $\cN_{\nu,\nu'}^{r}$ we have $\Gamma\equiv0$. Thus,
\begin{equation}\label{buse}\begin{split}
\int_{\RR^2\setminus\big(\cN_{\nu,\nu'}^{r}\cup\cP_{\nu,\nu'}^{r,p}\big)}\Big|\Gamma(\nu',r,w,L)&-\Gamma(\nu,r,w,L)\Big|\,\bigg|f_i^b\bigg(w,\frac{L}{r_b}\bigg)\bigg|\bigg(1+\mathbbm{1}_{w^2+\frac{L^2}{r_b^2}<2\mU_L}\bigg)\mathbbm{1}_{r \geq \mR_i(w,L)}\,dw\,dL
\\&\leq\int_{\RR^2\setminus\big(\cN_{\nu,\nu'}^{r}\cup\cP_{\nu,\nu'}^{r,p}\big)}\Big|\Gamma(\nu,r,w,L)\Big|\,\bigg|f_i^b\bigg(w,\frac{L}{r_b}\bigg)\bigg|\,dw\,dL.
\end{split}
\end{equation}
Therefore, the two cases~\eqref{faucon} and~\eqref{buse} reduces to study
\begin{equation}
    J_{\nu,\nu'}^{r}:=\int_{\RR^2\setminus\cP_{\nu,\nu'}^{r,p}}\frac{|w|}{\big|w^2-L^2\big(\frac{1}{r^2}-\frac{1}{r_b^2}\big)-2\nu'\big|^\frac{1}{2}}\,\bigg|f_i^b\bigg(w,\frac{L}{r_b}\bigg)\bigg|\,dw\,dL.
\end{equation}
By the Hölder inequality (with $q>2$ and $1/q+1/q'=1$), 
\begin{equation}\label{epervier}\begin{split}
    J_{\nu,\nu'}^{r}&\leq\bigg(\int_{\RR^2\setminus\cP_{\nu,\nu'}^{r,p}}\bigg|f_i^b\bigg(w,\frac{L}{r_b}\bigg)\bigg|\,dw\,dL\bigg)^\frac{1}{q}\bigg(\int_{\RR^2}\frac{|w|^{q'}}{\big|w^2-L^2\big(\frac{1}{r^2}-\frac{1}{r_b^2}\big)-2\nu'\big|^\frac{{q'}}{2}}\,\bigg|f_i^b\bigg(w,\frac{L}{r_b}\bigg)\bigg|\,dw\,dL\bigg)^\frac{1}{{q'}}\\
    &\leq \frac{}{2-q'}\bigg(\int_{\RR^2\setminus\cP_{\nu,\nu'}^{r,p}}\bigg|f_i^b\bigg(w,\frac{L}{r_b}\bigg)\bigg|\,dw\,dL\bigg)^\frac{1}{q}\,\Big(\|f_i^b\|_{L^1}+\|f_i^b\|_{L^1_L(L^\infty_w(w\,dw))}\Big)^\frac{1}{q'},
\end{split}
\end{equation}
where Lemma~\ref{lem:well defined g_i g_e} is used for the last inequality. The announced Hölder estimate is given by the study of
\begin{equation*}
    K_{\nu,\nu'}^{r,p}:=\int_{\RR^2\setminus\cP_{\nu,\nu'}^{r,p}}\bigg|f_i^b\bigg(w,\frac{L}{r_b}\bigg)\bigg|\,dw\,dL.
\end{equation*}
We now observe that
\begin{equation*}
    (w,L)\notin\cP_{\nu,\nu'}^{r,p}\quad\Longleftrightarrow\quad w^2-2\nu'-\frac{\nu-\nu'}{|w|^{2\frac{p-1}{3-p}}}<L^2\bigg(\frac{1}{r^2}-\frac{1}{r_b^2}\bigg)<w^2-2\nu'+\frac{\nu-\nu'}{|w|^{2\frac{p-1}{3-p}}}.
\end{equation*}
the Fubini theorem then gives:
\begin{equation}\label{oeuf}
    K_{\nu,\nu'}^{r,p}= 2\int_{-\infty}^{+\infty}\int_{(M^1_{w,\nu,\nu'})_+^{1/2}}^{(M^2_{w,\nu,\nu'})_+^{1/2}}\bigg|f_i^b\bigg(w,\frac{L}{r_b}\bigg)\bigg|\,dL\,dw,
\end{equation}
where
\begin{equation*}
    M^1_{w,\nu,\nu'}:=\bigg(\frac{1}{r^2}-\frac{1}{r_b^2}\bigg)^{-1}\bigg(w^2-2\nu'-\frac{\nu-\nu'}{|w|^{2\frac{p-1}{3-p}}}\bigg)\quad\text{and}\quad M^2_{w,\nu,\nu'}:=\bigg(\frac{1}{r^2}-\frac{1}{r_b^2}\bigg)^{-1}\bigg(w^2-2\nu'+\frac{\nu-\nu'}{|w|^{2\frac{p-1}{3-p}}}\bigg).
\end{equation*}
The number $2$ in factor of~\eqref{oeuf} comes from the use of the symmetry $f^b_i(w,L)=f^b_i(w-L)$. Equation~\eqref{oeuf} gives
\begin{equation*}\begin{split}
        K_{\nu,\nu'}^{r,p}&\leq 2\int_{-\infty}^{+\infty}\Big|(M^2_{w,\nu,\nu'})_+^{1/2}-(M^1_{w,\nu,\nu'})_+^{1/2}\Big|\sup_{L\in\RR}\big|f_i^b(w,L)\big|\,dw\\
        &\leq C(r)\sqrt{\nu-\nu'\,}\int_{-\infty}^{+\infty}\sup_{L\in\RR}\big|f_i^b(w,L)\big|\,\frac{dw}{|w|^\frac{p-1}{3-p}}=C\|f_i^b\|_{L^1_w(L^\infty_L\,;dw/|w|^\gamma)}\,\sqrt{\nu-\nu'\,},
    \end{split}
\end{equation*}
where for the last equality we used that $p$ has been chosen to have $\gamma=(p-1)/(3-p)$. 
The function $C(r)$ is equal (up to a multiplicative constant) to $1/(r^{-1}-r_b^{-1})^{1/2}$.
Plugging this estimate back into~\eqref{epervier} and choosing $q=2/(p-1)>2$ gives
\begin{equation}\label{piou piou}\begin{split}
        J_{\nu,\nu'}^{r}&\leq \frac{C(r)}{p-1}\big(\|f_i^b\|_{L^1}+\|f_i^b\|_{L^1_L(L^\infty_w(w\,dw))}\big)^\frac{3-p}{2}\,\|f_i^b\|_{L^1_w(L^\infty_L\,;dw/|w|^\gamma)}^\frac{p-1}{2}\,(\nu-\nu')^\frac{p-1}{4}\\
        &\leq\frac{C(r)}{\gamma}\big(\|f_i^b\|_{L^1}+\|f_i^b\|_{L^1_L(L^\infty_w(w\,dw))}\big)^\frac{1}{\gamma+1}\,\|f_i^b\|_{L^1_w(L^\infty_L\,;dw/|w|^\gamma)}^\frac{\gamma}{\gamma+1}\,(\nu-\nu')^\frac{\gamma}{2(\gamma+1)}.
        \end{split}
\end{equation}

\emph{Conclusion of the proof:} If we now gather the two estimates obtained respectively at Step 1. with~\eqref{vautour} and Step 2. with~\eqref{piou piou}, we get (using $\nu-\nu'\leq1$),
\begin{equation}
   \big| g_i(\nu,r)-g_i(\nu',r)\big|\,\leq\, \frac{C(r)}{\gamma(1-\gamma)}\Big(1+\|f_i^b\|_{L^1}+\|f_i^b\|_{L^1_L(L^\infty_w(w\,dw))}+\|f_i^b\|_{L^1_w(L^\infty_L\,;dw/|w|^\gamma)}\Big)\,(\nu-\nu')^\frac{\gamma}{2(\gamma+1)}.
\end{equation}
A similar reasoning works for the function $g_e$.
\end{proof}

\subsection{Existence with minimization argument}
It is a standard technique to build solution to Poisson equations when under semi-linear form~\eqref{eq:equivalent equation} with variational argument. Indeed, being a solution to~\eqref{eq:equivalent equation} is equivalent to being a critical point of the following functional:
\begin{equation}\label{def:cJ}
    \cJ(\psi):=\int_0^{1}\Bigg\{\frac{1}{2}\bigg|\frac{d\psi}{dx}(x)\bigg|^2-G\Big(\psi(x),\,x\Big)\Bigg\}\,dx,
\end{equation}
where $G(\nu,r):=\int_0^\nu g(s,r)\,ds.$ 

We now recall
\begin{equation*}
    H^1_0([0,1]):=\bigg\{\psi:[0,1]\to\RR\;:\;\psi(0)=a,\quad\psi(1)=0,\quad\text{and}\quad\int_0^1\bigg|\frac{d\psi}{dx}(x)\bigg|^2dx<+\infty\bigg\}.
\end{equation*}
The Poincaré inequality implies that $H^1_0([0,1])\subseteq L^2([0,1])$ and the Rellich-Kondrachov theorem states that this injection is compact. We are interested in the following minimization problem:
\begin{equation}\label{eq:minimization problem}
    \text{Does it exists }\psi^\star\in H^1_0([0,1])\text{ such that }\qquad\cJ(\psi^\star)=\inf\limits_{\psi\in H^1_0([0,1])}\,\cJ(\psi)\quad ?
\end{equation}

\begin{lemma}[Existence of a minimizer]\label{lem:existence of a max}
The function $\cJ$ satisfy the following inequality:
\begin{equation}\label{eq:cJ bound}
    \frac{1}{2}\int_0^1\bigg|\frac{d\psi}{dx}(x)\bigg|^2\,dx\;\leq2\;\cJ(\psi)+\frac{1}{2\pi}\|g\|_{L^\infty}^2.
\end{equation}

In consequence, the minimization problem~\eqref{eq:minimization problem} admits a solution $\psi^\star\in H^1_0([0,1])$ and this function is a solution of~\eqref{eq:equivalent equation}.
\end{lemma}
\begin{proof}First, we observe that
\begin{equation*}
    \int_0^1\Big|G\Big(\psi(x),\,x\Big)\Big|\,dx=\int_0^1\bigg|\int_0^{\psi(x)}g\big(\nu,\,x\big)\,d\nu\bigg|\,dx\leq\|g\|_{L^\infty(\RR\times[0,1])}\,\|\psi\|_{L^1([0,1])}\leq\|g\|_{L^\infty}\|\psi\|_{L^2}
\end{equation*}
where the last inequality is the Cauchy-Schwarz inequality.
We continue this estimate using the Young inequality (with $\varepsilon>0$) and the Poincaré inequality (the constant of Poincaré of $[0,1]$ being $1/\pi$) in that order:
\begin{equation}\label{elephant}\begin{split}
    \cJ(\psi)&\geq\frac{1}{2}\int_0^1\bigg|\frac{d\psi}{dx}(x)\bigg|^2\,dx-\|g\|_{L^\infty}\|\psi\|_{L^2}\\
    &\geq\frac{1}{2}\int_0^1\bigg|\frac{d\psi}{dx}(x)\bigg|^2\,dx-\frac{1}{4\varepsilon}\|g\|_{L^\infty}^2-\varepsilon\|\psi\|_{L^2}^2\\
    &\geq\bigg(\frac{1}{2}-\frac{\varepsilon}{\pi}\bigg)\int_0^1\bigg|\frac{d\psi}{dx}(x)\bigg|^2\,dx-\frac{1}{4\varepsilon}\|g\|_{L^\infty}^2.
    \end{split}
\end{equation}
The announced inequality~\eqref{eq:cJ bound} is then obtained by taking $\varepsilon=\pi/4$ in~\eqref{elephant}.

Consider now $(\psi_n)$, a sequence of functions belonging to $H^1_0([0,1])$ that is minimizing the studied quantity $\cJ$. 
Equation~\eqref{eq:cJ bound} implies that $d\psi_n/dx$ is a bounded sequence in $L^2$. 
Therefore there exists $\psi^\star\in H^1_0([0,1])$ such that, up to an omitted extraction,
\begin{equation}\label{eq:weak psi n}
    \frac{d\psi_n}{dx}\;\longrightarrow\;\frac{d\psi^\star}{dx},\qquad\text{weakly in } L^2,
\end{equation}
and, by compact embedding,
\begin{equation*}
    \psi_n\;\longrightarrow\;\psi^\star,\qquad\text{strongly in } L^2.
\end{equation*}
This last convergence result implies, using the Lebesgue dominated convergence theorem, 
\begin{equation*}
    \int_0^1G\Big(\psi_n(x),\,x\Big)\,dx\;\longrightarrow\;\int_0^1G\Big(\psi^\star(x),\,x\Big)\,dx,\qquad\text{as }n\to+\infty.
\end{equation*}
Moreover, the convergence~\eqref{eq:weak psi n}, since $\psi\mapsto\int_0^1|\psi|^2$ is convex on $H^1_0([0,1])$, gives
\begin{equation*}
    \int_0^1\bigg|\frac{d\psi^\star}{dx}(x)\bigg|^2dx\leq\liminf\limits_{n\to+\infty}\int_0^1\bigg|\frac{d\psi_n}{dx}(x)\bigg|^2dx.
\end{equation*}
These two facts together imply, since $\psi_n$ is a minimizing sequence for $\cJ$,
\begin{equation*}
    \cJ(\psi^\star)\leq\inf\limits_{\psi\in H^1_0([0,1])}\,\cJ(\psi),
\end{equation*}
which eventually gives the existence of a minimizer for $\cJ$. The function $\psi^\star$ satisfies Equation~\eqref{eq:equivalent equation} because, as a minimizer, it is a critical point of the functional $\cJ$.
\end{proof}

\subsection{Passing to the limit in the parameters}
We have now the existence result for Equation~\eqref{eq:equivalent equation}, and then for~\eqref{eq:studied equation}, for any choice of parameters $\mU_l$, $\mV_L$, $\mR_i(w,L)$ and $\mR_e(w,L)$. 
To conclude to the existence of a solution for the initial problem~\eqref{eq:Poisson studied}, there remain to adjust these parameters in the view of Section~\ref{sec:param}. 

\subsubsection{Study of the barrier parameters problem}
The idea to adjust the \emph{barrier parameters} $\mR_i(w,L)$ and $\mR_e(w,L)$ in such a way that for almost every $(w,L)\in\RR^2$,
\begin{equation}\label{eq:adjust parameters}
    \mR_i(w,L)=\widetilde{\rho}[U_L]\bigg(\frac{w^2}{2}+\frac{L^2}{2r_b^2}\bigg),\qquad\text{and}\qquad\mR_e(w,L)=\widetilde{\rho}[V_L]\bigg(\frac{w^2}{2}+\frac{L^2}{2r_b^2}\bigg),
\end{equation}
is to do a fixed-point procedure. 
For that purpose, we need to study more precisely $\widetilde{\rho}$ defined at~\eqref{def:tilde rho} to obtain continuity properties.

For $\phi:[1,r_b]\to\RR$ be a continuous function, we define
\begin{equation} \label{obel_transform}
    \phi^\dag(r):=\max\limits_{r'\in[r,r_b]}\phi(r').
\end{equation}
The function $\phi^\dag$ is the smallest non-increasing function such that $\phi^\dag\geq\phi$. 

\begin{lemma}\label{lem:tilde rho dag}
Let $e\in\RR$ and let $\phi:[1,r_b]\to\RR$ be a continuous function. We have
\begin{equation*}
    \widetilde{\rho}[\phi](e)=\widetilde{\rho}[\phi^\dag](e).
\end{equation*}
\end{lemma}
\begin{proof}
To start with, we recall that
$$ \widetilde{\rho}[\phi](e)=\min\big\{a\in[1,r_b]:\forall\,s\geq a,\,\phi(s)\leq e\big\}.$$
We point out that if $e\geq max\,\phi$ then, \begin{equation}\label{biche} \{a\in[1,r_b]:\forall\,s\geq a,\,\phi(s)\leq e\big\}=[1,r_b],\end{equation}
so that we have $\widetilde{\rho}[\phi](e)=1.$ In this situation we also have $e\geq max\,\phi=\phi^\dag(1)\geq\phi^\dag(r)$, where the last inequality is given by the monotony of $\phi^\dag$. Therefore~\eqref{biche} also hold for $\phi^\dag$ and then $\widetilde{\rho}[\phi^\dag](e)=1.$

We now focus on the case $e<max\,\phi$. This implies that $\widetilde{\rho}[\phi](e)>1$. For this case, we first observe that, since $\phi\leq\phi^\dag$, by definition of $\widetilde{\rho}$,
\begin{equation}\label{caribou}
    \widetilde{\rho}[\phi](e)\leq \widetilde{\rho}[\phi^\dag](e).
\end{equation}
For the reverse inequality, we start by observing that (by continuity of $\phi$) the definition of $\widetilde{\rho}$ is equivalent to the two following propositions:
\begin{equation}\label{renne}
    \forall\,r\geq\widetilde{\rho}[\phi](e),\qquad e\geq\phi(r),
\end{equation}
and
\begin{equation}\label{orignal}
    \exists\,\delta>0,\;\forall\,r\in\big[\widetilde{\rho}[\phi](e)-\delta;\widetilde{\rho}[\phi](e)\big],\qquad\phi(r)>e.
\end{equation}
Indeed,~\eqref{renne} holds for all the elements of the set $\{a\in[1,r_b]:\forall\,s\geq a,\,\phi(s)\leq e\big\}$ while~\eqref{orignal} characterizes the fact that $\widetilde{\rho}[\phi](e)$ is the smallest element of this set. By continuity and since $\widetilde{\rho}[\phi](e)>1$, Equations~\eqref{renne} and~\eqref{orignal} gives that,
\begin{equation}\label{faon}
    \phi\big(\widetilde{\rho}[\phi](e)\big)=e.
\end{equation}
Equations~\eqref{renne} and~\eqref{faon} together imply
\begin{equation*}
    \max_{r\geq\widetilde{\rho}[\phi](e)}\phi(r)=\phi\big(\widetilde{\rho}[\phi](e)\big).
\end{equation*}
Thus,
\begin{equation}\label{cerf}
    \phi^\dag\big(\widetilde{\rho}[\phi](e)\big) =\phi\big(\widetilde{\rho}[\phi](e)\big).
\end{equation}
On the other hand, since $\phi^\dag$ is non-increasing, Equation~\eqref{caribou} implies
\begin{equation}\label{chevreil}
    \phi^\dag\big(\widetilde{\rho}[\phi](e)\big)\geq\phi^\dag\big(\widetilde{\rho}[\phi^\dag](e)\big).
\end{equation}
Suppose now by the absurd that $\widetilde{\rho}[\phi^\dag](e)>\widetilde{\rho}[\phi](e)$, then~\eqref{orignal} and~\eqref{chevreil} (since $\phi^\dag$ is non-increasing) give
\begin{equation*}
\phi^\dag\big(\widetilde{\rho}[\phi](e)\big)>\phi^\dag\big(\widetilde{\rho}[\phi^\dag](e)\big)
\end{equation*}
This last inequation with~\eqref{faon} and~\eqref{cerf} lead to
\begin{equation*}
    e=\phi\big(\widetilde{\rho}[\phi](e)\big)=\phi^\dag\big(\widetilde{\rho}[\phi](e)\big)>\phi^\dag\big(\widetilde{\rho}[\phi^\dag](e)\big)=e,
\end{equation*}
which is eventually contradictory.
\end{proof}

We have also the following continuity property for the $\dag$ application:
\begin{lemma}[Application $\dag$ is Lipschitz]\label{lem:dag lip}
Let $\phi$ and $\psi$ be two continuous functions on $[1,r_b]$. We have
\begin{equation}\label{eq:dag lip}
    \|\phi^\dag-\psi^\dag\|_{L^\infty}\leq\|\phi-\psi\|_{L^\infty}.
\end{equation}
\end{lemma}
\begin{proof}
Let $r\in[1,r_b]$, we have
\begin{equation*}
    |\phi^\dag(r)-\psi^\dag(r)|=\Big|\max_{y\in[r,r_b]}\phi(y)-\max_{y\in[r,r_b]}\psi(y)\Big|\leq\max_{y\in[r,r_b]}|\phi(y)-\psi(y)|\leq\|\phi-\psi\|_{L^\infty}.
\end{equation*}
taking the $max$ at the left-hand side above gives~\eqref{eq:dag lip}.
\end{proof}
We are now in position to give the convergence result for the non-linearity $\widetilde{\rho}$:
\begin{lemma}[Convergence property for $\widetilde{\rho}$]\label{lem:tilde rho cv}
Let $(\phi_n)$ be a sequence of continuous functions that is uniformly converging towards $\phi$. Then for almost every $e\in\RR$,
\begin{equation*}
    \widetilde{\rho}[\phi_n](e)\;\longrightarrow\;\widetilde{\rho}[\phi](e).
\end{equation*}
\end{lemma}
\begin{proof}
Since we have $\phi_n\to\phi$ in $L^\infty$, then by Lemma~\ref{lem:dag lip} we have $\phi_n^\dag\to\phi^\dag$ in $L^\infty.$ 
Let $e\in\RR$, suppose that there exists $r\in[1,r_b]$ such that $\phi^\dag(r)>e$.
By uniform convergence, there exists $\delta>0$ such that for all $n\in\NN$ large enough: $\phi_n^\dag(r)\geq e+\delta.$ By definition of $\widetilde{\rho}$, we deduce that $r\leq\widetilde{\rho}[\phi_n^\dag](e)$.
In the view of Lemma~\ref{lem:tilde rho dag}, this gives $r\leq\widetilde{\rho}[\phi_n](e)$.
By taking the $lim\,inf$ we conclude:
\begin{equation*}
    \phi^\dag(r)>e\quad\Longrightarrow\quad r\leq \liminf_{n\to+\infty}\widetilde{\rho}[\phi_n](e).
\end{equation*}
Thus, with $\phi^\dag$ being non-increasing,
\begin{equation*}
    \inf\,\{r\in[1,r_b]:\phi^\dag(r)=e\}\leq\liminf_{n\to+\infty}\widetilde{\rho}[\phi_n](e).
\end{equation*}
Similarly,
\begin{equation*}
    \sup\,\{r\in[1,r_b]:\phi^\dag(r)=e\}\geq\limsup_{n\to+\infty}\widetilde{\rho}[\phi_n](e)
\end{equation*}
Since $\phi^\dag$ is non-increasing, if we have $meas\big\{r\in[1,r_b]:\phi^\dag(r)=e\big\}=0$ then this set is a singleton. 
In this case, the two estimates above give the convergence of $\widetilde{\rho}[\phi_n](e)$.

We now remark the following general fact: if $f:\RR^d\to\RR$ is a measurable function, then the set of $y\in\RR$ such that $meas\{x\in\RR^d:f(x)=y\}>0$ is a set of measure $0$. Indeed, using the layer-cake representation~\cite[chap.1]{Lieb_Loss_2001} (direct corollary of the Fubini theorem),
\begin{equation}
    \begin{split}
        0&=\int_{\RR^d}0\,dx=\int_{\RR^d}\meas\,\big\{y\in\RR:f(x)=y\big\}\,dx\\
        &=\int_{\RR^d\times\RR}\mathbbm{1}_{\{(x,y)\in\RR^d\times\RR\,:\,f(x)=y\}}\,dx\,dy\\
        &=\int_\RR\meas\big\{x\in\RR^d:f(x)=y\big\}\,dy.
    \end{split}
\end{equation}
From this we conclude that the set of $e\in\RR$ such that $meas\big\{r\in[1,r_b]:\phi^\dag(r)=e\big\}>0$ has indeed its measure equal to $0$ and therefore the announced convergence holds for almost every $e\in\RR$.
\end{proof}

\begin{corollary}\label{coro:tilde rho cv}
For almost every $(w,L)\in\RR^2$,
\begin{equation}\label{eq:tilde rho U_L cv}
    \widetilde{\rho}\bigg[\phi_n+\frac{L^2}{2\,\cdot\,^2}\bigg]\bigg(\frac{w^2}{2}+\frac{L^2}{2\,r_b^2}\bigg)\;\longrightarrow\;\widetilde{\rho}\bigg[\phi+\frac{L^2}{2\,\cdot\,^2}\bigg]\bigg(\frac{w^2}{2}+\frac{L^2}{2\,r_b^2}\bigg).
\end{equation}

\end{corollary}
\begin{proof}
Let $L\in\RR$ be fixed. If $\phi_n\to\phi$ in $L^\infty([1,r_b])$ then $\phi_n+\frac{L^2}{2\,\cdot^2}$ converges in $L^\infty$ to $\phi+\frac{L^2}{2\,\cdot^2}$. As a consequence of the previous lemma, the set of $w\in\RR$ such that~\eqref{eq:tilde rho U_L cv} does not hold is of measure $0$. Corollary~\ref{coro:tilde rho cv} then follows (using the Fubini theorem).
\end{proof}
Recall that $\phi(r)+\frac{L^2}{2r^2}=U_L(r)$ so that the convergence above is exactly the one needed to adjust the parameter $\mR_i(w,L)$ in the view of~\eqref{eq:adjust parameters}. The arguments are similar for $\mR_e(w,L)$ with $-\phi(r)+\frac{L^2}{2r^2}=V_L(r).$

\subsubsection{Passing to the limit with the parameters}
We can now consider passing to the limit with barrier-parameters and obtain~\eqref{eq:adjust parameters}. 
For that purpose, we suppose that the functions $f_i^b$ and $f_e^b$ are in $L^1\cap L^1_L(L^\infty_w(w\,dw))$ and also in $L^1_w(L^\infty_L\,;dw/|w|^\gamma)$ for some $0<\gamma<1$.

We also have to adjust the max-parameters to obtain
\begin{equation*}
    \mU_L=\overline{U_L}:=\max_{r\in[1,r_b]}\phi(r)+\frac{L^2}{2\,r^2},\qquad\text{and}\qquad\mV_L=\overline{V_L}:=\max_{r\in[1,r_b]}-\phi(r)+\frac{L^2}{2\,r^2}.
\end{equation*}
For that purpose, we proceed with an iterative fixed-point argument.
We construct sequences of parameters $(\mR_i^n(w,L))_{n\in\NN}$, $(\mR_e^n(w,L))_{n\in\NN},$ $(\mU_L^n)_{n\in\NN}$ and $(\mV_L^n(w,L))_{n\in\NN}$, a sequence of functions $g_n:\RR\times[0,1]\to\RR$, a sequence $\psi_n:[0,1]\to\RR$ and a sequence $\phi_n:[1,r_b]\to\RR$ as follows.
The first element of the sequences can be chosen freely without importance.
Suppose that for $n\in\NN$, we have already built the $n^{th}$ term of the sequences: $\mR_i^n(w,L)$, $\mR_e^n(w,L),$ $\mU_L^n$ and $\mV_L^n(w,L)$, $g_n,$ $\psi_n$ and $\phi_n$. We define for all $(w,L)\in\RR^2$,
\begin{equation}\label{def:mR_i n}
    \mR_i^{n+1}(w,L):=\widetilde{\rho}\bigg[\phi_n+\frac{L^2}{2\,\cdot\,^2}\bigg]\bigg(\frac{w^2}{2}+\frac{L^2}{2\,r_b^2}\bigg),
\end{equation}
\begin{equation}\label{def:mR_e n}
    \mR_e^{n+1}(w,L):=\widetilde{\rho}\bigg[-\phi_n+\frac{L^2}{2\,\cdot\,^2}\bigg]\bigg(\frac{w^2}{2}+\frac{L^2}{2\,r_b^2}\bigg),
\end{equation}
\begin{equation}\label{def:mU_L n}
    \mU_L^{n+1}:=\overline{U_L^n}=\max_{r\in[1,r_b]}\phi_n(r)+\frac{L^2}{2\,r^2},
\end{equation}
\begin{equation}\label{def:mV_L n}
    \mV_n^{n+1}:=\overline{V_L^n}=\max_{r\in[1,r_b]}-\phi_n(r)+\frac{L^2}{2\,r^2}.
\end{equation}
We now define $g_{n+1}:\RR\times[0,1]\to\RR$ using~\eqref{def:g} where the associated function $\widetilde{g}_{n+1}$ is defined by~\eqref{def:tilde g}\eqref{def:g_i}\eqref{def:g_e} with parameters $\mU_L^{n+1}$, $\mV_L^{n+1}$, $\mR_i^{n+1}(w,L)$ and $\mR_e^{n+1}(w,L)$. 
We now define the function $\psi_{n+1}:[0,1]\to\RR$ as being a minimizer on $H^1_0([0,1])$ of $\cJ$ defined at~\eqref{def:cJ} with function $G=G_{n+1}$ defined by $\int_0^xg_{n+1}(\nu,x')\,dx'$.
Such a minimizer exists and is solution to~\eqref{eq:equivalent equation}, as stated by Lemma~\ref{lem:existence of a max}. 

Note that there may exist infinitely many maximizers so that this step of the proof requires the axiom of choice.
From $\psi_{n+1}$ we define $\phi_{n+1}$ with~\eqref{eq:psi to phi} and $\phi_{n+1}$ is a solution to~\eqref{eq:studied equation} with function $\widetilde{g}_{n+1}$.

The sequences being well-defined, we study their limit.
The fact that $\psi_n$ satisfy~\eqref{eq:equivalent equation} implies in particular that for all $n\in\NN$,
\begin{equation}\label{eq:the bound}\begin{split}
&\qquad\bigg\|\frac{d^2}{dx^2}\psi_n\bigg\|_{L^\infty}\leq\|g_n\|_{L^\infty}\\&\leq C\,\sup_{\nu,r}\int_{\RR^2}\Gamma(\nu,r,w,L)f_i^b\bigg(w,\frac{L}{r_b}\bigg)\,dw\,dL+C\,\sup_{\nu,r}\int_{\RR^2}\Gamma\Big(\!-\nu,r,w,L\Big)f_e^b\bigg(w,\frac{L}{r_b}\bigg)\,dw\,dL\\
&\leq C\big(\|f_i^b\|_{L^1}+\|f_e^b\|_{L^1}+\|f_i^b\|_{L^1_L(L^\infty_w(w\,dw))}+\|f_e^b\|_{L^1_L(L^\infty_w(w\,dw))}\big)
\end{split}
\end{equation}
where the last estimate is given by Lemma~\ref{lem:well defined g_i g_e}.
In particular, $d^2\psi_n/dx^2$ is a bounded sequence in $L^\infty$.
By compact embedding, we obtain that, up to an omitted extraction of subsequence, the function $\psi_n$ converges in $H^1_0$ towards some function $\psi^\star$.
As a consequence, $\phi_n$ converges towards $\phi^\star$ where $\phi^\star$ is deduced from $\psi^\star$ with~\eqref{eq:psi to phi}.
By Sobolev embedding, the convergence of $\phi_n$ also takes place in $L^\infty$ and therefore Corollary~\ref{coro:tilde rho cv} implies
\begin{equation*}
    \widetilde{\rho}\bigg[\phi_n+\frac{L^2}{2\,\cdot\,^2}\bigg]\bigg(\frac{w^2}{2}+\frac{L^2}{2\,r_b^2}\bigg)\;\longrightarrow\;\widetilde{\rho}\bigg[\phi^\star+\frac{L^2}{2\,\cdot\,^2}\bigg]\bigg(\frac{w^2}{2}+\frac{L^2}{2\,r_b^2}\bigg)
\end{equation*}
for almost every $(w,L)\in\RR^2$. As a consequence of~\eqref{def:mR_i n}, we also have $\mR_i^{n}(w,L)$ converging for almost every $(w,L)\in\RR^2$ towards a limit $\mR_i^\star(w,L)$ and
\begin{equation*}
       \mR_i^\star(w,L)=\widetilde{\rho}\bigg[\phi^\star+\frac{L^2}{2\,\cdot\,^2}\bigg]\bigg(\frac{w^2}{2}+\frac{L^2}{2\,r_b^2}\bigg). 
\end{equation*}
Similarly,
\begin{equation*}
        \mR_e^n(w,L)\;\longrightarrow\;\mR_e^\star(w,L)=\widetilde{\rho}\bigg[-\phi^\star+\frac{L^2}{2\,\cdot\,^2}\bigg]\bigg(\frac{w^2}{2}+\frac{L^2}{2\,r_b^2}\bigg).
\end{equation*}
For almost every $(w,L)\in\RR^2$. Concerning the convergence of the max-parameters, we write
\begin{equation*}\begin{split}
    &\big|\overline{U_L^\star}-\overline{U_L^n}\big|=\bigg|\bigg(\max_{r\in[1,r_b]}\phi^\star(r)+\frac{L^2}{2\,r^2}\bigg)-\bigg(\max_{r\in[1,r_b]}\phi_n(r)+\frac{L^2}{2\,r^2}\bigg)\bigg|\\
    &\leq\max_{r\in[1,r_b]}\bigg|\bigg(\phi^\star(r)+\frac{L^2}{2\,r^2}\bigg)-\bigg(\phi_n(r)+\frac{L^2}{2\,r^2}\bigg)\bigg|=\|\phi^\star-\phi_n\|_{L^\infty}.
\end{split}
\end{equation*}
Thus, the convergence of $\phi_n$ towards $\phi^\star$ in $L^\infty$ implies the convergence of $\overline{U_L^n}$ to $\overline{U_L^\star}$ for all $L\in\RR$.
Using~\eqref{def:mU_L n}, we get
\begin{equation*}
    \mU_L^n\;\longrightarrow\; \mU_L^\star:=\overline{U_L^\star}=\max_{r\in[1,r_b]}\phi^\star(r)+\frac{L^2}{2\,r^2}.
\end{equation*}
A similar reasoning with~\eqref{def:mV_L n} gives the analogous result for $\mV_L^n$.

We now define $g^\star$ with~\eqref{def:g} where the chosen parameters are $\mU_L^\star$, $\mV_L^\star$, $\mR_i^\star(w,L)$ and $\mR_e^\star(w,L)$.
The Lebesgue dominated convergence theorem gives that for all $\nu,x$ we have $g_n(\nu,x)$ converging towards $g^\star(\nu,x)$. Invoking now Lemma~\ref{lem:Holder regularity}, we get that the family of functions $(\nu\mapsto g_n(\nu,x))_{n\in\NN}$ is uniformly equi-continuous for every fixed $x\in[0,1)$. Therefore, by Arzelà-Ascoli theorem, we have for all $x\in[0,1)$,
\begin{equation*}
    \sup_\nu\big|g_n(\nu,x)-g^\star(\nu,x)\big|\longrightarrow0,\quad\text{as }n\to+\infty.
\end{equation*}
Thus,
\begin{equation*}
    \forall\,x\in[0,1),\qquad g_n(\psi_n(x),x)\longrightarrow g^\star(\psi^\star(x),x).
\end{equation*}
Using again the bound~\eqref{eq:the bound}, we get that the convergence above also takes place in $L^2$.
Therefore, with the equation~\eqref{eq:equivalent equation}, we deduce that $d^2\psi_n/dx^2$ converges strongly in $L^2$ towards $d^2\psi^\star/dx^2$ and the following equality holds:
\begin{equation*}
     \forall\,x\in[0,1),\qquad -\frac{d^2\psi^\star}{dx^2}(x)=g^\star\Big(\psi^\star(x),\,x\Big).
\end{equation*}
Thus, $\psi^\star$ is solution to~\eqref{eq:equivalent equation} with function $g^\star$ and $\phi^\star$ is solution to~\eqref{eq:studied equation} with function $\widetilde{g}^\star$. 
Since the convergence of $(\psi_n)$ towards $\psi^\star$ takes place in $H^1_0$, the Dirichlet boundary conditions for $\psi^\star$ are satisfied and so is the case for $\phi^\star$.

\begin{corollary}
The Langmuir problem written in term of Poisson equation~\eqref{eq:Poisson studied} admits a solution and therefore the initial Langmuir-Vlasov-Poisson problem~\eqref{VPBVP} admits a weak-strong solution in the sense given by Definition~\ref{def_sol}.
\end{corollary}
~\newline

\noindent{\large\textbf{Acknowledgments}}\vspace{0.2cm}

The authors of this article acknowledge grant support from the project "Multiéchelle et Trefftz pour le transport numérique" (ANR-19-CE46-0004) of the Agence Nationale de la Recherche (France).

\appendix
\begin{lemma}[Countability of the locus of left strict local maxima]
\label{countability_max}
Let $f : \RR \rightarrow \RR$ be a function. Let 
\begin{equation} \label{def_A_f}
    A_{f} := \lbrace a \in \RR \: : \exists \delta > 0, \ \forall x \in (a-\delta,a)\: f(x) < f(a) \rbrace.
\end{equation} 
Then $A_{f}$ is at most countable.
\begin{proof}
If $A_{f}$ is empty the conclusion follows. Otherwise, let $a \in A_{f}.$ By definition, there exists $\delta_{a} > 0$ such that for all $ x\in (a - \delta_{a},a),$ $f(x) < f(a).$ It is equivalent to the existence of $n_{a} \in \NN^{*}$ such that for all $ x\in (a - \frac{1}{n_{a}}, a )$, $f(x) < f(a).$ One then considers the map $a \in A_{f} \mapsto n_{a}$. Therefore one has $A_{f} = \underset{n \in \NN^{*}}\bigcup A_{n}$ where $A_{n} := \lbrace a \in A_{f} : n_{a} = n \rbrace.$  Let $n \in \mathbb{N}^{*}$. If $a,a' \in A_{n}$ are such that $a \neq a'$ then necessarily $(a'-a)\textnormal{sgn}(a'-a) \geq \frac{1}{n}.$ Otherwise this would yield that $f(a) < f(a')$ and $f(a)> f(a')$ and one would get a contradiction. Invoking the density of $\QQ$ in $\RR$, for each $a \in A_{n}$ one can choose a  rational number $p_{a}$ such that $a - \frac{1}{2n} < p_{a} <a.$ Then for each $a \neq a'$ the number $p_{a}$ and $p_{a'}$ are distinct because $(a'-a)\textnormal{sgn}(a'-a) \geq \frac{1}{n}.$ Therefore the map $ a \in A_{n} \mapsto p_{a} \in \QQ$ is injective and thus $A_{n}$ is at most countable by countability of $\QQ$. Eventually $A_{f}$ is at most countable as the the union of at most countable sets.
 \end{proof}
\end{lemma}
\begin{proposition}[Additional properties for the transformation $\dag$] \label{properties_obel}
Let $p \in [1,+\infty]$ and $\phi \in W^{1,p}(1,r_{b})$. Consider $\phi^{\dag}$ defined by \eqref{obel_transform} and the set
\[
A_{\phi}^{\dag} := \lbrace b \in (1,r_{b}) \: : \: \exists \delta > 0, \: \forall x \in (b-\delta,b)\:, \phi(x) < \phi(b) \text{ and }  \phi^{\dag}(b) = \phi(b) \rbrace.
\]
One has then has following:
\begin{itemize}
\item[ a)] $\phi^{\dag}$ is continuous in $[1,r_{b}]$. 
\item[b)] Let $ 1\leq a < b \leq r_{b}$ such that $\phi^{\dag} - \phi > 0$ on $(a,b).$ Then $\phi^{\dag}$ is constant on $(a,b)$.
\item[c)] $\lbrace x \in (1,r_{b}) \: : \: \phi^{\dag}(x) -\phi(x) > 0 \rbrace = \underset{n \in I}{\cup} (a_{n},b_{n})$ where $(b_{n})_{n \in I}$ is a bijection from a subset $I \subseteq \NN$ into $A^{\dag}_{\phi}$ and the sequence $(a_{n})_{n \NN}$ is given by 
\begin{align*}
    \forall n \in I, \quad a_{n} := \inf \lbrace a \in (1,r_{b}) \: : \: \forall x \in (a,b_{n}), \phi(x) < \phi^{\dag}(b_{n}) \rbrace. 
\end{align*}
Moreover, for all $ n \in I$ such that $\phi^{\dag}(b_{n})$ is not the maximum value of $\phi$, one has $\phi(a_{n}) = \phi^{\dag}(a_{n}) = \phi^{\dag}(b_{n}).$ The intervals $((a_{n},b_{n}))_{n \in I}$ are disjoints.

\item [d)] $\phi^{\dag} \in W^{1,p}(1,r_{b})$, $(\phi^{\dag})' = \underset{ \lbrace \phi^{\dag} = \phi \rbrace }{\mathbbm{1}} \phi'$, and $\| (\phi^{\dag})' \|_{L^{p}} \leq \| \phi' \|_{L^{p}}.$ 
\end{itemize}
\begin{proof} 
a) Let $x,y \in [1,r_{b}]$ and assume without loss of generality that $x < y.$ The function $\phi^{\dag}$ being non increasing, one has
\begin{align*}
    \left \vert \phi^{\dag}(x)- \phi^{\dag}(y) \right \vert = \left \vert \underset{ x' \in [x,r_{b}]}{\max \phi(x')} - \underset{ x'' \in [y,r_{b}] }{\max \phi(x'')} \right \vert  = \underset{ x' \in [x,r_{b}]}{\max \phi(x')} - \underset{ x'' \in [y,r_{b}] }{\max \phi(x'')} .
\end{align*}
If $\underset{[x,r_{b}]}{\max}\ \phi = \underset{[y,r_{b}]}{\max}\ \phi$ then the difference in the above equality vanishes.  Otherwise, one has $\underset{[x,r_{b}]}{\max}\ \phi > \underset{[y,r_{b}]}{\max}\ \phi$ and therefore $\underset{[x,r_{b}]}{\max}\ \phi = \underset{[x,y]}{\max}\ \phi$. It yields,
  \begin{align*}
   &\left \vert \phi^{\dag}(x)- \phi^{\dag}(y) \right \vert = \underset{x'\in [x,y]}{\max}\ \phi(x') - \underset{ x'' \in [y,r_{b}] }{\max \phi(x'')} \leq  \underset{x'\in [x,y]}{\max}\ \phi(x') - \phi(y) \leq  \underset{x'\in [x,y]}{\max}\ \left( \phi(x')- \phi(y)\right),
\end{align*}
where one has used the fact that $\phi(y) \leq \underset{ x'' \in [y,r_{b}] }{\max \phi(x'')}$. The conclusion then follows from the continuity of $\phi.$ \newline\newline
b) Let $1 \leq a < b \leq r_{b}$ such that for all $x \in (a,b),$ $\phi(x) < \phi^{\dag}(x).$  Moving $b$ if necessary, one assumes that $\phi(b) = \phi^{\dag}(b).$ One shows that for all $x \in (a,b),$ $\phi^{\dag}(x) := \underset{ x' \in [x,r_{b}]}{\max} \ \phi(x') = \underset{ x' \in [b,r_{b}] }{ \max}\phi(x') =: \phi^{\dag}(b)$. Assume for the sake of the contradiction it is not the case. Then there is $x \in (a,b)$ such that $\underset{ x' \in [x,r_{b}]}{\max} \ \phi(x') > \underset{ x' \in [b,r_{b}] }{ \max}\phi(x')$. Therefore there is $ c \in (x, b)$ such that $\phi(c)> \underset{ x' \in [b,r_{b}] }{ \max}\phi(x') = \phi^{\dag}(b) = \phi(b).$ One can thus consider the point $c$ given by $c = \underset{r \in [x,b]}{\arg \max} \ \phi(r)$ (this point exists by continuity of $\phi$). At this point, one has
$
    \phi(c) = \underset{ [x,b]}{\max} \ \phi = \underset{ [c,b]}{\max} \ \phi = \underset{ [c,r_{b}]}{\max} \ \phi
$
where the last equality holds because $\phi(c) > \underset{ x' \in [b,r_{b}] }{ \max}\phi(x').$ One eventually remarks that by definition one has $\underset{ [c,r_{b}]}{\max} \ \phi = \phi^{\dag}(c)$ and thus $\phi(c) = \phi^{\dag}(c)$ which yields a contradiction.
\newline\newline
c) In virtue of Lemma \ref{countability_max}, the  set of points in $(1,r_{b})$ that are strict left local maxima of $\phi$ is at most countable so is the case for the subset $A_{\phi}^{\dag}$. Therefore there exists  a bijection $b: I \rightarrow A_{\phi}^{\dag}$ where $I \subseteq \NN.$ One now justifies the existence of the sequence $(a_{n})_{n \in I}.$ For each $n \in I$ the set $\lbrace a \in (1,r_{b}) \: : \: \forall x \in (a,b_{n}) \: \ \phi(x) < \phi^{\dag}(b_{n}) \rbrace$ is not empty since $b_{n}$ corresponds to a strict local maxima of $\phi$ that is $\phi^{\dag}(b_{n}) =  \phi(b_{n}).$ Since it is moreover lower bounded, the infimum exists. Therefore the sequence $(a_{n})_{n \in \NN}$ is well-defined. Since $\phi^{\dag}(b_{n})$ is not a maximum value of $\phi.$, by continuity of the function $\phi$, one has $\phi(a_{n}) = \phi^{\dag}(b_{n})$. Using the property a) and b), $\phi^{\dag}$ is constant in the interval $[a_{n},b_{n}]$, one has then $\phi^{\dag}(a_{n}) = \phi(a_{n}) = \phi^{\dag}(b_{n})$. One now proves that the intervals $((a_{n},b_{n}))_{n \in I}$ are disjoints. If $n,m \in I$ are such that $n \neq m$ then $b_{n} \neq b_{m}$ because $b$ is bijective. One assumes without loss of generality that $b_{n} < b_{m}$. Then necessarily  $b_{n} \leq a_{m},$ otherwise if $b_{n} > a_{m}$, one has one the one hand $\phi(b_{n}) < \phi^{\dag}(b_{m}) = \phi(b_{m})$ and on the other hand $\phi(a_{m}) < \phi^{\dag}(b_{n}) = \phi(b_{n}).$ But one has also by definition $\phi(a_{m}) = \phi^{\dag}(b_{m}) = \phi(b_{m})$, therefore one has both $\phi(b_{m}) < \phi(b_{n})$ and $\phi(b_{n}) < \phi(b_{m})$ which is a contradiction, thus $ b_{n} \leq a_{m}$.  Consequently, the open intervals $(a_{n},b_{n})$ are disjoints. One shows the equality of the sets.
By definition of the sequences $(a_{n})_{n \in I}$ and $(b_{n})_{n \in I}$ one has $\underset{n \in I}{\cup} (a_{n},b_{n}) \subset \lbrace x \in (1,r_{b}) \: : \: \phi^{\dag}(x) -\phi(x) > 0 \rbrace.$  For the reverse embedding, one takes $x \in (1,r_{b})$ such that $\phi^{\dag}(x) > \phi(x).$ By continuity there exists $ 1 \leq a < x < b \leq r_{b}$ such that  for all $y \in (a,b)$, $\phi^{\dag}(y) > \phi(y).$  Therefore consider the two numbers
\begin{align*}
  a^{*} = \inf \lbrace a' \leq a \: : \: \phi^{\dag}(y) > \phi(y) \:  \forall y \in (a',x) \rbrace,\\
  b^{*} = \sup \lbrace b' \geq b \: : \: \phi^{\dag}(y) > \phi(y) \: \forall y \in (x,b') \rbrace.
\end{align*}
By continuity of the function $\phi^{\dag}-\phi$, one has $\phi^{\dag}(a^{*})= \phi(a^{*})$ and $\phi^{\dag}(b^{*}) = \phi(b^{*}).$ Moreover, using the point a) and b), $\phi^{\dag}$ is constant on the interval $[a^{*},b^{*}]$. Therefore for all $y\in [a^{*},b^{*}]$, $\phi^{\dag}(y) = \phi^{\dag}(b^{*}) = \phi(b^{*})$. Thus, it implies that  for all $y \in (a^{*},b^{*})$, $\phi(y) < \phi^{\dag}(y) = \phi^{\dag}(b^{*}) = \phi(b^{*})$ thus $b^{*} \in A_{\phi}^{\dag}$. Since the set $A_{\phi}^{\dag}$ is at most countable there exists $n \in I$ such that $b_{*} = b_{n}.$ By construction one also has $a^{*} = a_{n}$ which shows that $ \lbrace x \in (1,r_{b}) \: : \: \phi^{\dag}(x) - \phi(x) > 0 \rbrace \subset \underset{n \in I}{\cup} (a_{n},b_{n}).$

d) Using the point a) $\phi^{\dag}$ is a continuous function on the compact set $[1,r_{b}]$, it is therefore bounded and thus in $L^{p}(1,r_{b}).$ Let $\psi \in C^{\infty}_{c}(1,r_{b})$, then one has
\begin{align*}
    \int_{1}^{r_{b}} \phi^{\dag}(x) \psi'(x) dx = \int_{ \lbrace \phi^{\dag}- \phi > 0 \rbrace} \phi^{\dag}(x) \psi'(x)dx + \int_{\lbrace \phi^{\dag} = \phi \rbrace} \phi(x) \psi'(x)dx.
\end{align*}
Using the point c), one has $\lbrace \phi^{\dag}- \phi > 0 \rbrace = \underset{ n \in I} \cup (a_{n},b_{n})$ where $I \subseteq \NN$ and the two sequences $(a_{n})_{n \in I}$ and $(b_{n})_{n \in I}$ are such that $a_{n} < b_{n}$, $\phi^{\dag}(a_{n}) = \phi(a_{n}) = \phi(b_{n}) = \phi^{\dag}(b_{n})$ for all $n \in I$.  If $I$ is finite then $\lbrace \phi^{\dag}- \phi > 0 \rbrace$ is a finite union of disjoints intervals. The conclusion then follows after decomposing the integral into a finite sum of integrals on each intervals and using integration by parts. If $I$ is not finite then $I= \NN$ and $\lbrace \phi^{\dag}- \phi > 0 \rbrace$ is countable union of the disjoint intervals  $(a_{n},b_{n})$. One therefore obtains
\begin{align*}
    \int_{ \lbrace \phi^{\dag}- \phi > 0 \rbrace} \phi^{\dag}(x) \psi'(x)dx = \sum_{n \in \NN} \int_{a_{n}}^{b_{n}} \phi^{\dag}(x) \psi'(x)dx,
\end{align*}
where the above sum is convergent because it is absolutely convergent. Indeed for $N \in \NN$, the partial sum $S_{N} = \displaystyle \sum_{n = 0}^{N} \int_{a_{n}}^{b_{n}} \vert \phi^{\dag}(x) \psi'(x) \vert dx$ is non decreasing and upper bounded: for all $N \in \NN$, $S_{N}\leq \int_{1}^{r_{b}} \vert \phi^{\dag}(x) \psi'(x) \vert dx < +\infty. $  
Using the fact that $\phi^{\dag}$ is constant in the interval $[a_{n},b_{n}]$, one has
\begin{align*}
    \int_{\lbrace \phi^{\dag}-\phi > 0 \rbrace} \phi^{\dag}(x) \psi'(x)dx = \sum_{n \in \NN} \phi^{\dag}(b_{n}) (\psi(b_{n})-\psi(a_{n})).
\end{align*}
On the complementary set  $\lbrace \phi^{\dag} = \phi \rbrace = \underset{n \in \NN} \cap (1,r_{b}) \setminus (a_{n},b_{n})$, one has also
\begin{align}
&\int_{\lbrace \phi^{\dag} = \phi \rbrace} \phi(x) \psi'(x)dx = \left( \sum_{n \in \NN} \phi^{\dag}(b_{n})( \psi(a_{n}) - \psi(b_{n})) \right) - \int_{\lbrace \phi^{\dag} = \phi \rbrace} \phi'(x) \psi(x)dx.
\end{align}
Gathering the two integrals together, the boundary terms eventually cancel and one obtains
\begin{align*}
    \int_{1}^{r_{b}} \phi^{\dag}(x) \psi'(x)dx = - \int_{1}^{r_{b}} \mathbbm{1}_{\lbrace \phi^{\dag} = \phi \rbrace}(x) \phi'(x) \psi(x)dx.
\end{align*}
Since $\phi'$ is in $L^{p}(1,r_{b})$ so is the case for the function $\mathbbm{1}_{\lbrace \phi^{\dag} = \phi \rbrace} \phi'$. One thus deduces that $\phi^{\dag}$ is in $W^{1,p}(1,r_{b})$ and that its weak derivative is given almost everywhere in $(1,r_{b})$ by $(\phi^{\dag})' = \mathbbm{1}_{\lbrace \phi^{\dag} = \phi \rbrace} \phi'.$ One therefore easily gets the inequality $\| (\phi^{\dag})' \|_{L^{p}} \leq \| \phi' \|_{L^{p}}.$ It concludes the proof.
\end{proof}
\end{proposition}
\section*{Appendix: Proof of proposition \ref{prop:weak_sol}}
\label{appendix}

\begin{proof}
Let $f_{i}^{b}$ be an essentially bounded function, therefore $f_{i}$ defined by \eqref{def_fi} belongs to $L^{1}_{\textnormal{loc}}(Q)$. Let $\psi \in C^{1}(\overline{Q})$  compactly supported on $\overline{Q}$ and such that $\psi_{| \Sigma^{\textnormal{out}} }= 0$. Consider the function $\Psi$ defined for all $(r,v_{r},v_{\theta}) \in Q$ by 
\begin{align*}
   \Psi(r,v_{r},v_{\theta}) = v_{r} \partial_{r} \psi(r,v_{r},v_{\theta}) +  \left( \frac{v_{\theta}^{2}}{r} - \partial_{r} \phi(r) \right) \partial_{v_{r}} \psi(r,v_{r},v_{\theta}) - \frac{v_{r}}{r} \partial_{v_{\theta}} (v_{\theta} \psi)(r,v_{r},v_{\theta})
\end{align*}
where the function $\phi$ is in the space $W^{2,\infty}(1,r_{b}).$
One has using the Fubini theorem,
\begin{align*}
\int_{Q}  \Psi(r,v_{r},v_{\theta}) f_{i}(r,v_{r},v_{\theta})dv_{r}dv_{\theta} dr =\int_{1}^{r_{b}} \int_{\RR} \int_{\RR}\Psi(r,v_{r},v_{\theta}) f_{i}(r,v_{r},v_{\theta})dv_{\theta}dv_{r} dr.
\end{align*}
Using the change variable $L = r v_{\theta}$ in the integral with respect to $v_{\theta}$ one obtains,
\begin{align*}
    &\int_{Q}  \Psi(r,v_{r},v_{\theta}) f_{i}(r,v_{r},v_{\theta})dv_{r}dv_{\theta} dr= \int_{1}^{r_{b}} \int_{\RR} \int_{\RR}\Psi \left( r,v_{r},\frac{L}{r} \right) f_{i}\left( r,v_{r},\frac{L}{r} \right) \frac{1}{r} dL dv_{r} dr\\
    &=  \int_{-\infty}^{+\infty} \int_{[1,r_{b}] \times \RR} \frac{1}{r} \Psi\left( r,v_{r},\frac{L}{r}\right) f_{i}\left( r,v_{r},\frac{L}{r} \right) dv_{r}  dr dL. \\
\end{align*}
For $L \in \RR$ being fixed, the function $(r,v_{r})\mapsto  f_{i}(r,v_{r},L)$ vanishes on $\cD_{i}^{pc}(L)$, one therefore has
\begin{align*}
      &\int_{Q}  \Psi(r,v_{r},v_{\theta}) f_{i}(r,v_{r},v_{\theta})dv_{r}dv_{\theta} dr\\
      &= \int_{-\infty}^{+\infty}\int_{\cD_{i}^{b}(L)}\frac{1}{r} \Psi\left( r,v_{r},\frac{L}{r}\right) f_{i}^{b}\left(-\sqrt{v_{r}^2 +2\left( U_{L}(r)-U_{L}(r_{b}) \right) }, \frac{L}{r_{b}} \right) dv_{r}  dr dL\\
     & = \int_{-\infty}^{+\infty}\int_{\cD_{i}^{b,1}(L)}\frac{1}{r} \Psi\left( r,v_{r},\frac{L}{r}\right) f_{i}^{b}\left(-\sqrt{v_{r}^2 +2\left( U_{L}(r)-U_{L}(r_{b}) \right) }, \frac{L}{r_{b}} \right) dv_{r}  dr dL\\
     & +\int_{-\infty}^{+\infty}\int_{\cD_{i}^{b,2}(L)}\frac{1}{r} \Psi\left( r,v_{r},\frac{L}{r}\right) f_{i}^{b}\left(-\sqrt{v_{r}^2 +2\left( U_{L}(r)-U_{L}(r_{b}) \right) }, \frac{L}{r_{b}} \right) dv_{r}  dr dL
\end{align*}
where the sets $\cD_{i}^{b,1}(L)$ and $\cD_{i}^{b,2}(L)$ are defined respectively in \eqref{d_i_b_1} and \eqref{d_i_b_2}.
To continue the computation one considers for $(r,L) \in (1,r_{b}) \times \RR$ the two sets of radial velocities
\begin{align*}
    \cD_{i}^{b,1}(r,L) := \left \lbrace v_{r} \in \RR \: : \: v_{r} < - \sqrt{2(\overline{U}_{L} -U_{L}(r))} \right \rbrace,\\
    \cD_{i}^{b,2}(r,L) := \bigg \lbrace v_{r} \in  \RR \: : U_{L}(r_{b}) < \frac{v_{r}^2}{2} + U_{L}(r)< \overline{U_{L}} \:  , \: r > r_{i}\left( L,\frac{v_{r}^2}{2} + U_{L}(r) \right) \bigg \rbrace.
    \end{align*}
    For each couple $(r,L)$, these sets amount to pick the radial velocities that are on characteristics originating from the boundary $r = r_{b}$. 
One thus obtains 
\begin{align*}
    &\int_{Q}  \Psi(r,v_{r},v_{\theta}) f_{i}(r,v_{r},v_{\theta})dv_{r}dv_{\theta} dr\\
    &= \underbrace{ \int_{1}^{r_{b}}\int_{-\infty}^{+\infty} \int_{\mathcal{D}_{i}^{b,1}(r,L)}\frac{1}{r} \Psi\left( r,v_{r},\frac{L}{r}\right) f_{i}^{b}\left(-\sqrt{v_{r}^2 +2\left( U_{L}(r)-U_{L}(r_{b}) \right) }, \frac{L}{r_{b}} \right) dv_{r}  dL dr}_{:=I_{1}} \\
    &+\underbrace{\int_{1}^{r_{b}}\int_{-\infty}^{+\infty} \int_{\mathcal{D}_{i}^{b,2}(r,L)}\frac{1}{r} \Psi\left( r,v_{r},\frac{L}{r}\right) f_{i}^{b}\left(-\sqrt{v_{r}^2 +2\left( U_{L}(r)-U_{L}(r_{b}) \right) }. \frac{L}{r_{b}} \right) dv_{r}dL  dr}_{:=I_{2}}.
\end{align*}
To ease the reading, one sets for $(r,L) \in (1,r_{b}) \times \RR,$
\begin{align*}
    I_{1}(r,L) := \int_{\mathcal{D}_{i}^{b,1}(r,L)}\frac{1}{r} \Psi\left( r,v_{r},\frac{L}{r}\right) f_{i}^{b}\left(-\sqrt{v_{r}^2 +2\left( U_{L}(r)-U_{L}(r_{b}) \right) }, \frac{L}{r_{b}} \right) dv_{r},\\
    I_{2}(r,L) := \int_{\mathcal{D}_{i}^{b,2}(r,L)}\frac{1}{r} \Psi\left( r,v_{r},\frac{L}{r}\right) f_{i}^{b}\left(-\sqrt{v_{r}^2 +2\left( U_{L}(r)-U_{L}(r_{b}) \right) }, \frac{L}{r_{b}} \right) dv_{r}.
\end{align*}
One first computes $I_{1}$, so let $(r,L) \in (1,r_{b}) \times \RR$, one has 
\begin{align*}
    I_{1}(r,L) = \int_{-\infty}^{-\sqrt{2(\overline{U}_{L}-U_{L}(r))}} \frac{1}{r} \Psi\left( r,v_{r},\frac{L}{r}\right) f_{i}^{b}\left(-\sqrt{v_{r}^2 +2\left( U_{L}(r)-U_{L}(r_{b}) \right) }, \frac{L}{r_{b}} \right) dv_{r}.
\end{align*}
Using the change of variable $w_{r} = -\sqrt{v_{r}^2 +2\left( U_{L}(r)-U_{L}(r_{b}) \right) }$ yields
\begin{align*}
    I_{1}(r,L) = \int_{-\infty}^{-\sqrt{2(\overline{U}_{L}-U_{L}(r_{b}))}}\frac{1}{r} \frac{\Psi\left(r, -\sqrt{w_{r}^2 -2\left( U_{L}(r)-U_{L}(r_{b})\right)} , \frac{L}{r} \right)}{-\sqrt{w_{r}^2 -2\left( U_{L}(r)-U_{L}(r_{b})\right)}}f_{i}^{b}\left( w_{r},\frac{L}{r_{b}} \right) w_{r} dw_{r}.
\end{align*}
The integrand in $I_{1}$ has an apparent singularity at each point $r \in (1,r_{b})$ such that $U_{L}(r) = \overline{U}_{L}$. This singularity is integrable because the product $\Psi f_{i}^{b}$ is bounded. To go further, one considers for $(w_{r},L) \in \RR^2$ such that $w_{r} <- \sqrt{2(\overline{U}_{L}-U_{L}(r_{b}))}$, the restriction of the function $\psi$ to a characteristic curve of equation $v_{r} = \pm \sqrt{ w_{r}^2 -2(U_{L}(r)-U_{L}(r_{b}))}$. Then, we  set
\begin{align} \label{psi_pm}
    \psi^{\pm} : r \in (1,r_{b}) \mapsto \frac{1}{r} \psi \left( r,\pm \sqrt{w_{r}^2 -2\left( U_{L}(r)-U_{L}(r_{b})\right)} , \frac{L}{r} \right). 
\end{align}
Using the chain rule, one verifies that for all $r \in (1,r_{b}),$ 
\begin{align} \label{chain_rule_id}
\frac{d}{dr} \left(\frac{1}{r} \psi^{\pm}\right)(r)  =  \frac{1}{r} \frac{\Psi\left(r, \pm \sqrt{w_{r}^2 -2\left( U_{L}(r)-U_{L}(r_{b})\right)} ; \frac{L}{r} \right)}{\pm \sqrt{w_{r}^2 -2\left( U_{L}(r)-U_{L}(r_{b})\right)}}.
\end{align}
One therefore obtains (permuting the derivative and the integral) that
\begin{align*}
    I_{1}(r,L)= \frac{d}{dr} \left( \int_{-\infty}^{-\sqrt{2(\overline{U}_{L}-U_{L}(r_{b}))}}\frac{1}{r} \psi^{-}(r) f_{i}^{b}\left( w_{r},\frac{L}{r_{b}} \right) w_{r} dw_{r} \right).
\end{align*}
After an integration with respect to $L$ and with respect to $r$, one eventually gleans
\begin{align*}
 &I_{1} = \int_{1}^{r_{b}} \frac{d}{dr} \left( \int_{-\infty}^{+\infty}\int_{-\infty}^{-\sqrt{2(\overline{U}_{L}-U_{L}(r_{b}))}} \frac{1}{r} \psi^{-}(r)f_{i}^{b}\left( w_{r},\frac{L}{r_{b}} \right) w_{r} dw_{r} dL \right) dr\\
 = &\int_{-\infty}^{+\infty}\int_{-\infty}^{-\sqrt{2(\overline{U}_{L}-U_{L}(r_{b}))}} \frac{1}{r_{b}} \psi^{-}(r_{b})f_{i}^{b}\left( w_{r},\frac{L}{r_{b}} \right) w_{r} dw_{r} dL\\
 = & \int_{-\infty}^{+\infty}\int_{-\infty}^{-\sqrt{2(\overline{U}_{L}-U_{L}(r_{b}))}} \frac{1}{r_{b}}\psi\left( r_{b},w_{r},\frac{L}{r_{b}} \right) f_{i}^{b}\left( w_{r},\frac{L}{r_{b}} \right) w_{r}dw_{r} dL.
\end{align*}
where one has used the fact that $\psi^{-}(1) = 0$ because $\psi$ vanishes on $\Sigma^{\textnormal{out}}.$ One deals with the computation of $I_{2}$. One sees that $I_{2}$ splits as
\begin{align*}
    I_{2} = \int_{1}^{r_{b}} \int_{-\infty}^{+\infty} \mathbbm{1}_{\lbrace U_{L}(r_{b}) - U_{L}(r) < 0 \rbrace} I_{2}(r,L) dL dr + \int_{1}^{r_{b}} \int_{-\infty}^{+\infty} \mathbbm{1}_{\lbrace U_{L}(r_{b}) - U_{L}(r) \geq 0 \rbrace} I_{2}(r,L) dL dr.
\end{align*}
For the sake of conciseness, one restricts the computation in the case where for all $L \in \RR$,  $\displaystyle U_{L}(r) > U_{L}(r_{b})$  for all $r \in (1,r_{b})$. The other case can be treated with similar computations. So consider
\begin{align*}
   I_{2} = \int_{1}^{r_{b}} \int_{-\infty}^{+\infty} \int_{\cD_{i}^{b,2}(r,L)}  \frac{1}{r} \Psi\left( r,v_{r},\frac{L}{r}\right) f_{i}^{b}\left(-\sqrt{v_{r}^2 +2\left( U_{L}(r)-U_{L}(r_{b}) \right) }, \frac{L}{r_{b}} \right) dv_{r}dL  dr
\end{align*}
where
\[
\cD_{i}^{b,2}(r,L) = \bigg \lbrace v_{r} \in  \RR \: : \vert v_{r} \vert < \sqrt{2\left( \overline{U}_{L} -U_{L}(r) \right)} \:  , \: r > r_{i}\left( L,\frac{v_{r}^2}{2} + U_{L}(r) \right) \bigg \rbrace.
\]
One recalls that this set is associated with characteristics curves that originates from $r = r_{b}$ and go back to $r = r_{b}.$
One remarks that the condition $r > r_{i}\left(L, \frac{v_{r}^{2}}{2} + U_{L}(r) \right)$ is equivalent to $U_{L}^{\dag}(r) \leq \frac{v_{r}^2}{2} + U_{L}(r)$ where $U_{L}^{\dag}$ is smallest non increasing function such that $U_{L}^{\dag} \geq U_{L}$. It is in particular given by \eqref{obel_transform}.
Therefore one has,
\[
\cD_{i}^{b,2}(r,L) = \bigg \lbrace v_{r} \in  \RR \: : \vert v_{r} \vert < \sqrt{2\left( \overline{U}_{L} -U_{L}(r) \right)} \:  , \: \vert v_{r} \vert  \geq \sqrt{2\left( U_{L}^{\dag}(r)-U_{L}(r) \right) }  \bigg \rbrace.
\]
One decomposes this set into $\mathcal{D}_{i}^{b,2}(r,L) = \mathcal{D}_{i}^{b,2,+}(r,L) \cup \mathcal{D}_{i}^{b,2,-}(r,L)$
with
\begin{align*}
    \cD_{i}^{b,2,+}(r,L) = \bigg \lbrace v_{r} \in  \RR \: : \sqrt{2\left( U_{L}^{\dag}(r)-U_{L}(r) \right) }  \leq  v_{r} < \sqrt{2\left( \overline{U}_{L} -U_{L}(r) \right)} \:  \bigg \rbrace,\\
    \cD_{i}^{b,2,-}(r,L) = \bigg \lbrace v_{r} \in  \RR \: :  -\sqrt{2\left( \overline{U}_{L} -U_{L}(r) \right)} < v_r \leq - \sqrt{2\left( U_{L}^{\dag}(r)-U_{L}(r) \right) }  \bigg \rbrace.
\end{align*}
Using the change of variable $w_{r} = -\sqrt{v_{r}^2 +2\left( U_{L}(r)-U_{L}(r_{b}) \right)}$ one gets
\begin{align*}
   & I_{2} = \int_{1}^{r_{b}}\int_{-\infty}^{+\infty}\int_{-\sqrt{2\left(\overline{U}_{L}-U_{L}(r_{b})\right)}}^{-\sqrt{2\left(U_{L}^{\dag}(r)-U_{L}(r_{b}) \right)}} \frac{\Psi\left(r, -\sqrt{w_{r}^2 -2\left( U_{L}(r)-U_{L}(r_{b})\right)} , \frac{L}{r} \right)}{-\sqrt{w_{r}^2 -2\left( U_{L}(r)-U_{L}(r_{b})\right)}}f_{i}^{b}\left( w_{r},\frac{L}{r_{b}} \right) w_{r} dw_{r}dL dr \\
    &- \int_{1}^{r_{b}}\int_{-\infty}^{+\infty}\int_{-\sqrt{2\left(\overline{U}_{L}-U_{L}(r_{b})\right)}}^{-\sqrt{2\left(U_{L}^{\dag}(r)-U_{L}(r_{b}) \right)}} \frac{\Psi\left(r, \sqrt{w_{r}^2 -2\left( U_{L}(r)-U_{L}(r_{b})\right)} , \frac{L}{r} \right)}{\sqrt{w_{r}^2 -2\left( U_{L}(r)-U_{L}(r_{b})\right)}}f_{i}^{b}\left( w_{r},\frac{L}{r_{b}} \right) w_{r} dw_{r}dL dr.
\end{align*}
Using again the identity \eqref{chain_rule_id}, one obtains
\begin{align*}
    &I_{2} =       \int_{1}^{r_{b}}\int_{-\infty}^{+\infty} \int_{-\sqrt{2\left(\overline{U}_{L}-U_{L}(r_{b})\right)}}^{-\sqrt{2\left(U_{L}^{\dag}(r)-U_{L}(r_{b}) \right)}} \frac{d}{dr} \left( \frac{1}{r} ( \psi^{-}  - \psi^{+})(r) \right) f_{i}^{b}\left(w_{r},\frac{L}{r_{b}} \right) w_{r} dw_{r} dL dr.\\
\end{align*}
One now justifies the regularity of $U^{\dag}_{L}$ in order to use the chain rule. Since $\phi$ belongs to $W^{2,\infty}(1,r_{b})$, it belongs in particular to $W^{1,\infty}(1,r_{b})$. Therefore for all $L \in \RR$, the function $U_{L}$ is in the space $W^{1,\infty}(1,r_{b})$. One can thus apply the properties d) of Lemma \ref{properties_obel} with $p = +\infty.$ So one has $U_{L}^{\dag} \in W^{1,\infty}(1,r_{b})$. Since moreover, for all $r \in (1,r_{b})$, $U_{L}(r) > U_{L}(r_{b})$, one has also $U_{L}^{\dag}(r) > U_{L}(r_{b})$. Thus, for each $L \in \RR$, one obtains using the chain rule that for almost every $r \in (1,r_{b}),$
\begin{align*}
    &\frac{d}{dr} \int_{-\sqrt{2\left(\overline{U}_{L}-U_{L}(r_{b})\right)}}^{-\sqrt{2\left(U_{L}^{\dag}(r)-U_{L}(r_{b}) \right)}}\frac{1}{r} ( \psi^{-}  - \psi^{+})(r) f_{i}^{b}\left(w_{r},\frac{L}{r_{b}} \right) w_{r} dw_{r} = \\
    &\frac{ -(U_{L}^{\dag})'(r) }{r \sqrt{2\left( U_{L}^{\dag}(r) - U_{L}(r_{b}) \right)}}
    \left[ \psi \left( r,-\sqrt{2(U_{L}^{\dag}(r)-U_{L}(r))}, \frac{L}{r} \right) - \psi \left( r,\sqrt{2(U_{L}^{\dag}(r)-U_{L}(r))},\frac{L}{r} \right) \right]\\
    &+\int_{-\sqrt{2\left(\overline{U}_{L}-U_{L}(r_{b})\right)}}^{-\sqrt{2\left(U_{L}^{\dag}(r)-U_{L}(r_{b}) \right)}} \frac{d}{dr} \left( \frac{1}{r} ( \psi^{-}  - \psi^{+})(r) \right) f_{i}^{b}\left(w_{r},\frac{L}{r_{b}} \right) w_{r}dw_{r}.
\end{align*}
One remarks that the first term, which is a product, vanishes almost everywhere in $(1,r_{b})$: in the set where $U_{L}^{\dag}$ and $U_{L}$ are equal, the term in brackets vanishes because the difference vanishes. In the complementary set, $(U_{L}^{\dag})'$  vanishes almost everywhere because of the property d) of Lemma \ref{properties_obel}. Thus, integrating with respect to $L$ and $r$ one gets
\begin{align*}
  &I_{2} = \int_{1}^{r_{b}} \frac{d}{dr} \int_{-\infty}^{+\infty}  \int_{-\sqrt{2\left(\overline{U}_{L}-U_{L}(r_{b})\right)}}^{-\sqrt{2\left(U_{L}^{\dag}(r)-U_{L}(r_{b}) \right)}}\frac{1}{r} ( \psi^{-}  - \psi^{+})(r) f_{i}^{b}\left(w_{r},\frac{L}{r_{b}} \right) w_{r} dw_{r} dL dr. \\
\end{align*}
The integration with respect to $r$ eventually gives only the boundary term at $r = r_{b}$ because the other one vanishes since $\psi$ vanishes on $\Sigma^{\textnormal{out}}.$ One eventually gleans
\begin{align*}
    I_{2}= \int_{-\infty}^{+\infty}\int_{-\sqrt{2\left(\overline{U}_{L}-U_{L}(r_{b})\right)}}^{0}\frac{1}{r_{b}}\psi\left( r_{b},w_{r},\frac{L}{r_{b}} \right) f_{i}^{b}\left( w_{r},\frac{L}{r_{b}} \right) w_{r}dw_{r} dL\\
\end{align*}
where one uses the equality $U_{L}^{\dag}(r_{b}) = U_{L}(r_{b})$ and the fact that $\psi_{| \Sigma^{\textnormal{out}}} = 0$.
Gathering the integrals $I_{1}$ and $I_{2}$ together, one eventually concludes
\begin{align*}
    &\int_{Q}  \Psi(r,v_{r},v_{\theta}) f_{i}(r,v_{r},v_{\theta})dv_{r}dv_{\theta} dr = I_{1} +I_{2} = \int_{-\infty}^{+\infty}\int_{-\infty}^{0} \frac{1}{r_{b}}\psi\left( r_{b},w_{r},\frac{L}{r_{b}} \right) f_{i}^{b}\left( w_{r},\frac{L}{r_{b}} \right) w_{r}dw_{r} dL \\
    &=  \int_{-\infty}^{+\infty} \int_{-\infty}^{0} \psi\left( r_{b},w_{r},v_{\theta} \right) f_{i}^{b}\left( w_{r},v_{\theta} \right)w_{r}dw_{r} dv_{\theta}.
\end{align*}
\end{proof}

\bibliographystyle{elsarticle-harv}
\bibliography{references}
\end{document}

%% file: probe_drawing.pdf_t
\begin{picture}(0,0)%
\includegraphics{probe_drawing.pdf}%
\end{picture}%
\setlength{\unitlength}{3947sp}%
\begingroup\makeatletter\ifx\SetFigFont\undefined%
\gdef\SetFigFont#1#2#3#4#5{%
  \reset@font\fontsize{#1}{#2pt}%
  \fontfamily{#3}\fontseries{#4}\fontshape{#5}%
  \selectfont}%
\fi\endgroup%
\begin{picture}(8330,8328)(2736,-8725)
\put(6601,-3736){\makebox(0,0)[lb]{\smash{{\SetFigFont{20}{24.0}{\familydefault}{\mddefault}{\updefault}{\color[rgb]{0,0,0}$y$}%
}}}}
\put(6683,-4822){\makebox(0,0)[lb]{\smash{{\SetFigFont{20}{24.0}{\familydefault}{\mddefault}{\updefault}{\color[rgb]{0,0,0}$O$}%
}}}}
\put(7596,-4761){\makebox(0,0)[lb]{\smash{{\SetFigFont{20}{24.0}{\rmdefault}{\mddefault}{\updefault}{\color[rgb]{0,0,0}$x$}%
}}}}
\put(9451,-2461){\makebox(0,0)[lb]{\smash{{\SetFigFont{20}{24.0}{\familydefault}{\mddefault}{\updefault}{\color[rgb]{0,0,0}$Fe_{r}$}%
}}}}
\put(10346,-1411){\makebox(0,0)[lb]{\smash{{\SetFigFont{20}{24.0}{\familydefault}{\mddefault}{\updefault}{\color[rgb]{0,0,0}$e_{r}$}%
}}}}
\put(9376,-1176){\makebox(0,0)[lb]{\smash{{\SetFigFont{20}{24.0}{\familydefault}{\mddefault}{\updefault}{\color[rgb]{0,0,0}$e_{\theta}$}%
}}}}
\put(8251,-1616){\makebox(0,0)[lb]{\smash{{\SetFigFont{20}{24.0}{\familydefault}{\mddefault}{\updefault}{\color[rgb]{0,0,0}$\mathbbm{v}$}%
}}}}
\put(7801,-3736){\makebox(0,0)[lb]{\smash{{\SetFigFont{20}{24.0}{\familydefault}{\mddefault}{\updefault}{\color[rgb]{0,0,0}$r_{p}$}%
}}}}
\put(10201,-2086){\makebox(0,0)[lb]{\smash{{\SetFigFont{20}{24.0}{\familydefault}{\mddefault}{\updefault}{\color[rgb]{0,0,0}$r_{b}$}%
}}}}
\end{picture}%

%% file: potential_draw.pdf_t
\begin{picture}(0,0)%
\includegraphics{potential_draw.pdf}%
\end{picture}%
\setlength{\unitlength}{3947sp}%
\begingroup\makeatletter\ifx\SetFigFont\undefined%
\gdef\SetFigFont#1#2#3#4#5{%
  \reset@font\fontsize{#1}{#2pt}%
  \fontfamily{#3}\fontseries{#4}\fontshape{#5}%
  \selectfont}%
\fi\endgroup%
\begin{picture}(4469,8101)(1896,-8494)
\put(2026,-3586){\makebox(0,0)[lb]{\smash{{\SetFigFont{14}{16.8}{\rmdefault}{\mddefault}{\updefault}{\color[rgb]{0,0,0}$v_{r}$}%
}}}}
\put(6350,-2169){\makebox(0,0)[lb]{\smash{{\SetFigFont{14}{16.8}{\rmdefault}{\mddefault}{\updefault}{\color[rgb]{0,0,0}$U_{L} = e$}%
}}}}
\put(5932,-3612){\makebox(0,0)[lb]{\smash{{\SetFigFont{14}{16.8}{\rmdefault}{\mddefault}{\updefault}{\color[rgb]{0,0,0}$r$}%
}}}}
\put(4999,-3564){\makebox(0,0)[lb]{\smash{{\SetFigFont{14}{16.8}{\rmdefault}{\mddefault}{\updefault}{\color[rgb]{0,0,0}$r(L,e)$}%
}}}}
\put(5893,-6284){\makebox(0,0)[lb]{\smash{{\SetFigFont{14}{16.8}{\rmdefault}{\mddefault}{\updefault}{\color[rgb]{0,0,0}$r$}%
}}}}
\put(1922,-1308){\makebox(0,0)[lb]{\smash{{\SetFigFont{14}{16.8}{\rmdefault}{\mddefault}{\updefault}{\color[rgb]{0,0,0}$\overline{U}_{L}$}%
}}}}
\put(1911,-624){\makebox(0,0)[lb]{\smash{{\SetFigFont{14}{16.8}{\rmdefault}{\mddefault}{\updefault}{\color[rgb]{0,0,0}$U_{L}$}%
}}}}
\end{picture}%